\documentclass[reqno,twoside]{article}

\usepackage{amssymb}
\usepackage{amsmath}
\usepackage{amsthm}
\usepackage{letterswitharrows}
\usepackage{mathrsfs}
\usepackage{graphicx}
\usepackage{latexsym}
\usepackage{stmaryrd}
\usepackage{enumitem}
\setlist{topsep=0.3em, itemsep=-0.3em}
\usepackage{moreenum}
\usepackage{authblk}
\usepackage[bookmarks,hyperfootnotes=false, psdextra=true]{hyperref}
\usepackage{multirow}
\usepackage[dvipsnames]{xcolor}
\usepackage{caption}
\colorlet{darkishRed}{red!60!black}
\colorlet{darkishBlue}{blue!60!black}
\colorlet{darkishGreen}{green!50!black}
\colorlet{lightishGreen}{green!70!black}
\hypersetup{
    draft = false,
    bookmarksopen=true,
    colorlinks,
    linkcolor={darkishBlue},
    citecolor={lightishGreen},
    urlcolor={darkishBlue}
}
\usepackage[nameinlink, capitalise, noabbrev]{cleveref}
\usepackage{thmtools}
\usepackage{nameref}
\crefformat{enumi}{#2#1#3}
\crefformat{equation}{#2(#1)#3}
\crefrangeformat{section}{Sections~#3#1#4--#5#2#6}
\crefname{mainresult}{Theorem}{Theorems}
\let\setminus=\smallsetminus
\usepackage{tikz}
\usepackage{tikz-cd}
\usetikzlibrary{calc,through,intersections,arrows, trees, positioning, decorations.pathmorphing, cd}
\usepackage{comment}
\usepackage{mathtools}
\usepackage{nccmath}
\usepackage{pifont}
\usepackage[utf8]{inputenc}
\usepackage[T1]{fontenc}
\usepackage{lmodern}
\usepackage[babel]{microtype}
\usepackage[english]{babel}
\usepackage{relsize}

\usepackage{caption}

\usepackage{geometry}
\let\setminus=\smallsetminus
\renewcommand{\leq}{\leqslant}
\renewcommand{\geq}{\geqslant}

\let\rho=\varrho
\let\phi=\varphi

\newcommand{ \N } { \mathbb{N} }
\newcommand{ \Z } { \mathbb{Z} }

\newcommand{\defn}[1]{{\color{darkishRed}{\emph{#1}}}}
\newcommand{\defnm}[1]{{\color{darkishRed}{#1}}}

\makeatletter

\def\calCommandfactory#1{%
   \expandafter\def\csname c#1\endcsname{\mathcal{#1}}}
\def\frakCommandfactory#1{%
   \expandafter\def\csname frak#1\endcsname{\mathfrak{#1}}}

\newcounter{ctr}
\loop
  \stepcounter{ctr}
  \edef\X{\@Alph\c@ctr}
  \expandafter\calCommandfactory\X
  \expandafter\frakCommandfactory\X
  \edef\Y{\@alph\c@ctr}
  \expandafter\frakCommandfactory\Y
\ifnum\thectr<26
\repeat

\setenumerate{label={\normalfont (\roman*)}}

\newtheorem{theorem}{Theorem}[section] 
\newtheorem{proposition}[theorem]{Proposition}
\newtheorem{corollary}[theorem]{Corollary}
\newtheorem{lemma}[theorem]{Lemma}

\newtheorem{conjecture}[theorem]{Conjecture}

\newtheorem{mainresult}{Theorem}
\newtheorem{maincorollary}[mainresult]{Corollary}
\newtheorem{mainconjecture}[mainresult]{Conjecture}

\newtheorem{claim}{Claim}
\crefname{claim}{Claim}{Claims}
\AtEndEnvironment{proof}{\setcounter{claim}{0}}

\newenvironment{claimproof}{\noindent\textit{Proof.}}{\hfill\ensuremath{\blacksquare}\medskip}
\usepackage{etoolbox}

\newtheorem{innercustomthm}{}

\newenvironment{customthm}[1]
  {\renewcommand\theinnercustomthm{#1}\innercustomthm}
  {\endinnercustomthm}

\theoremstyle{definition}

\theoremstyle{remark}

\newcommand{\labtequtag}[3]{%
 \begin{equation} \label{#1} 	\begin{minipage}[c]{0.9\textwidth}  #2 \end{minipage} \ignorespacesafterend \tag{#3} \end{equation} }

\usepackage{etoolbox}
\newbool{pdfBool}
\booltrue{pdfBool} 

\newcommand{\pdfOrNot}[2]{\ifbool{pdfBool}{{#1}}{{#2}}}

\usepackage{svg}

\newbool{arXiv}
\boolfalse{arXiv} 

\newcommand{\arXivOrNot}[2]{\ifbool{arXiv}{{#1}}{{#2}}}

\usepackage{csquotes}
\usepackage[backend=biber, style=alphabetic, sorting=nyt, maxnames=99, maxalphanames=99, giveninits=true, sortcites=true]{biblatex}
\usepackage[skip=6pt]{subcaption} 
\let\tilde=\widetilde

\DeclareMathOperator{\kappatwo}{\left\lfloor \frac{\kappa -- 2}{2}\right\rfloor}

\usepackage{authblk}

\usepackage{fancyhdr}
\renewcommand\footnotemark{}

\begin{document}

	\title{A coarse Menger theorem for hyperbolic graphs, finitely presented groups, and more}

	\author{Sandra Albrechtsen$^{{\includegraphics[height=.6\baselineskip]{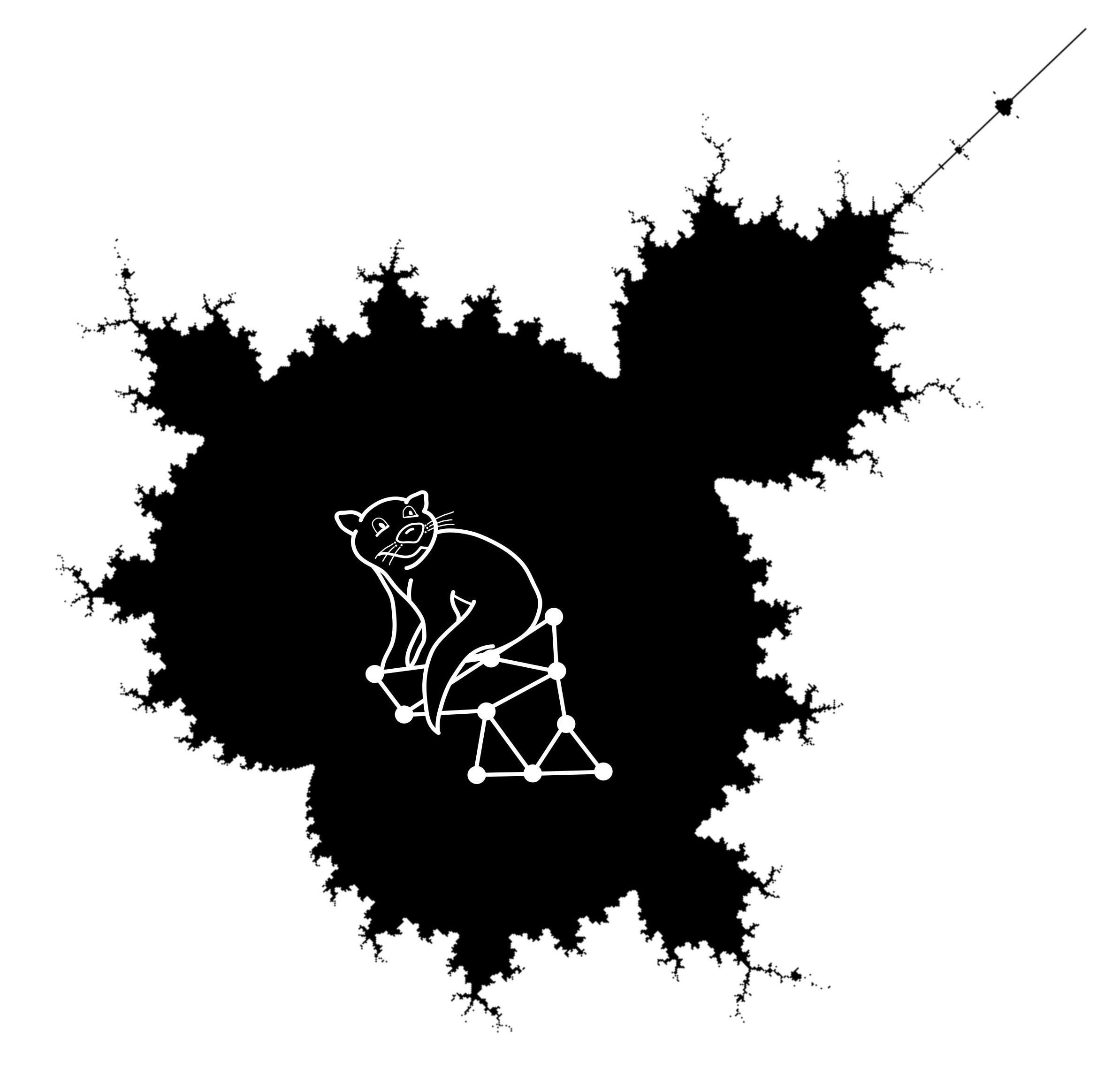}}}$}
	
	\affil{Institute of Mathematics, Leipzig University, Augustusplatz 10, 04109 Leipzig, Germany\\ sandra@albrechtsen-mail.de}

	\date{\vspace{-5ex}}
	
	\thanks{${\includegraphics[height=.8\baselineskip]{MandelbrotOtterSmall.jpg}}$ Supported by the Alexander von Humboldt Foundation in the framework of the Alexander von Humboldt Professorship of Daniel Král' endowed by the Federal Ministry of Education and Research.}

	\maketitle
	
	\begin{abstract}
		Menger's theorem is one of the most fundamental results in graph theory. It states that if a graph~$G$ does not contain $k$ disjoint paths between two given sets~$X$ and~$Y$ of vertices in~$G$, then there is a set of at most $k-1$ vertices that intersects every path between $X$ and~$Y$. 
		Nguyen, Scott, and Seymour gave a counterexample to the conjectured natural coarse variant in which the paths are required to be pairwise at distance at least~$d$, and, conversely, there is a set of at most $k-1$ bounded-radius balls intersecting every path between $X$ and~$Y$.
		In other words, the coarse Menger property does not hold in general.
		
		We prove that graphs whose cycles space is generated by cycles of bounded length do have the coarse Menger property.
		As a corollary, we show that many natural graphs and geodesic metric spaces have the coarse Menger property.
		These include hyperbolic graphs, Cayley graphs of finitely presented groups, planar graphs with bounded face size, and complete Riemannian planes.
	\end{abstract}
	
	

	\section{Introduction}

	Menger’s theorem~\cite{Menger1927} is one of the most fundamental results in graph theory. Given a graph $G$ and subsets $X, Y$ of  $V(G)$, it asserts that the maximum number of pairwise disjoint $X$--$Y$ paths is equal to the minimal size of a set of vertices intersecting every $X$--$Y$ path. This result is particularly attractive as it asserts that the obvious obstruction to having many disjoint $X$--$Y$ paths (a small set of vertices separating $X$ and $Y$) in fact \emph{must} appear in any graph that fails to have many such paths.
	This basic duality has sparked many further duality results and has led to countless applications in graph theory.
	
	Recently, there has been great interest in \emph{coarse graph theory}, a rapidly developing new area that studies graphs from a geometric perspective. It focuses on `large-scale' properties of graphs and aims to characterise graphs in terms of quasi-isometries (for an introduction to this area, see \cite{GP23}). One central direction of study is the attempt to lift results from structural graph theory to the coarse world, where statements about the structure of a graph are only made up to quasi-isometry, and objects are required to be far apart or close instead of disjoint or intersecting, respectively. 
	For this, a coarse analogue of Menger's theorem would be a powerful tool. Motivated by this, Albrechtsen, Huynh, Jacobs, Knappe, and Wollan~\cite{DistanceMengerForTwo} and, independently, Georgakopoulos and Papasoglu~\cite{GP23} conjectured the following coarse version of Menger's theorem:
	
	\begin{mainconjecture}[Coarse Menger Conjecture] \label{conj:CoarseMenger}
		There is a function $g: \N^2 \to \N$ such that for every $k,d \in \N$, for every graph (or geodesic metric space)~$G$ and every two sets $X,Y \subseteq V(G)$ at least one of the following statements holds:
		\begin{enumerate}
			\item \label{itm:CoarseMengerConj:Paths} There are $k$ disjoint $X$--$Y$ paths in $G$ that are pairwise at distance at least $d$ from each other.
			\item \label{itm:CoarseMengerConj:Set} There is a set $Z \subseteq V(G)$ of size at most $k-1$ such that every $X$--$Y$ path is at distance at most~$g(k,d)$ from~$Z$.
		\end{enumerate} 
	\end{mainconjecture}
	
	Let us say more generally that a class of graphs or geodesic metric spaces $\cG$ has the \defn{coarse Menger property} if there is a function $g: \N^2 \to \N$ (which is allowed to depend on $\cG$) such that \cref{conj:CoarseMenger} holds for every $G \in \cG$ with function~$g$.
	If there also exists a function $f:\N^2\to \N$ (which is allowed to depend on~$\cG$) such that every $G \in \cG$ satisfies \cref{conj:CoarseMenger} with function~$g$ if we allow the set $Z$ in \ref{itm:CoarseMengerConj:Set} to have size $f(k,d)$ instead of $k-1$, then we say that $\cG$ has the \defn{weak coarse Menger property}.
	
	Albrechtsen, Huynh, Jacobs, Knappe, and Wollan~\cite{DistanceMengerForTwo} and, independently, Georgakopoulos and Papasoglu~\cite{GP23} proved that \cref{conj:CoarseMenger} holds true for two paths (i.e.\ $k=2$).
	However, unfortunately, Nguyen, Scott, and Seymour~\cite{CoarseMengerCounterex} disproved the coarse Menger conjecture~\cite{DistanceMengerForTwo,GP23}, that is the class of all graphs does not have the coarse Menger property, and it already fails in the case of three paths. They \cite{NewCounterexToCoarseMenger} later improved their construction to even show that the class of all graphs does not have the weak coarse Menger property. 
	\smallskip
	
	Despite this, the coarse Menger property holds for some special classes of graphs and geodesic metric spaces.
	For instance, graphs of bounded path-width \cite{NSSMengerPathWidth} and series-parallel graphs \cite{NSSMengerPathWidth} both have the coarse Menger property.
	Although the coarse Menger property remains open for planar graphs, it is known in the restricted case that $X, Y$ are contained in the same face boundary \cite{NSSPathsAcrossADisc}. 
	\smallskip
	
	Nguyen, Scott, and Seymour's~\cite{CoarseMengerCounterex} coarse Menger counterexample has tree-width at most 6.
	So, already relatively simple classes of graphs do not have the coarse Menger property.
	However, it can be easily shown that graphs of bounded tree-width have the \emph{weak} coarse Menger property. Moreover, quite recently, other classes of graphs have been shown to have the weak coarse Menger property, including planar graphs, string graphs, and graphs embeddable on surfaces of bounded Euler genus~\cite{BPPCoarseMenger}, and, more generally, graphs excluding some fixed finite graph $H$ as a minor~\cite{CHLCoarseMenger}. In particular, this implies that also complete Riemannian surfaces of bounded Euler genus have the coarse Menger property~\cite{CHLCoarseMenger}.
	The induced variant of the weak coarse Menger property has also been studied for bounded-degree graphs and for graphs excluding a topological minor~\cite{InducedMengerBounded1,InducedMengerBounded2}.
	\smallskip
	
	In this paper, we show that several further natural classes of graphs and geodesic metric spaces have the (exact) coarse Menger property.
	Our main result is for graphs whose cycle space is generated by cycles of bounded length.
	We refer the reader to \cref{subsec:CycleSpace} for the definitions concerning the cycle space. Roughly speaking, the cycles space of a graph~$G$ is generated by cycles of length at most~$\kappa$ if every cycle of~$G$ can be obtained from cycles of~$G$ of length at most~$\kappa$ by repeated symmetric differences. 
	
	\begin{mainresult} \label{main:CoarseMenger:BoundedCycleSpace}
		For every $\kappa \in \N$, the class of graphs whose cycle space is generated by cycles of length at most~$\kappa$ has the coarse Menger property.
	\end{mainresult}
	
	\cref{main:CoarseMenger:BoundedCycleSpace} has several applications.
	For instance, as corollaries, we show that hyperbolic spaces, Cayley graphs of finitely presented groups, planar graphs with bounded face size, and complete Riemannian planes, all have the coarse Menger property.
	\smallskip
	
	A graph $G$ is \defn{$\delta$-hyperbolic} for some $\delta \in \N$ if for every three distinct vertices $u,v,w$ of $G$ and for every three shortest paths $P,Q,W$, one between each two of the vertices $u,v,w$, each of the paths $P,Q,W$ lies in the ball of radius $\delta$ around the union of the other two paths.
	
	\begin{maincorollary} \label{maincor:CoarseMenger:HyperbolicGraphs}
		For every $\delta > 0$, the class of $\delta$-hyperbolic graphs (or geodesic metric spaces) has the coarse Menger property.
	\end{maincorollary}
	
	In particular, \cref{maincor:CoarseMenger:HyperbolicGraphs} applies to Cayley graphs of hyperbolic groups.
	
	\begin{proof}
		It follows immediately from the definition of $\delta$-hyperbolic that a $\delta$-hyperbolic graph~$G$ does not contain any geodesic cycles of length at least~$6\delta+1$. Since the cycle space of a graph is generated by its geodesic cycles~\cite[Ch.~1 Exercise~44~(ii)]{Bibel}, the statement follows from \cref{main:CoarseMenger:BoundedCycleSpace} with $\kappa = 6 \delta + 1$.
		
		If $G$ is a $\delta$-hyperbolic geodesic metric space, then it is quasi-isometric to a graph (where the parameters of the quasi-isometry do not depend on the space) \cite[Corollary~21]{BDAsDim}, and the cycle space of this graph is generated by cycles of bounded length. The assertion follows as the coarse Menger property is preserved by quasi-isometries~\cite[Lemma~9.1]{CHLCoarseMenger}.
	\end{proof}

	\begin{maincorollary} \label{maincor:CoarseMenger:FintelyPresentedGroup}
		For every $\ell \in \N$, the class of locally finite Cayley graphs of finitely generated groups that have a finite presentation with relations of length at most~$\ell$ has the coarse Menger property.
	\end{maincorollary}
	
	In particular, \cref{maincor:CoarseMenger:FintelyPresentedGroup} applies to Cayley graphs of finitely generated, planar groups (which are finitely presented by a result of Droms~\cite{CDPlanarGroupsAreFinitelyPresented}).
	
	\begin{proof}
		It is well-known (see e.g.\ \cite[Section~4.1]{HAccessibilityInTransGraphs}) that the cycle space of a locally finite Cayley graph of a group $\Gamma$ that has a finite presentation $\langle S | R \rangle$ with relations of length at most~$\ell$ is generated by cycles of length at most $\ell$. Therefore, the statement follows from \cref{main:CoarseMenger:BoundedCycleSpace} with $\kappa = \ell$.
	\end{proof}

	It is conjectured \cite{NSSPathsAcrossADisc} that planar graphs have the coarse Menger property, and as discussed, this conjecture has received significant attention.
	Since the cycle space of a planar graph is generated by its (internal) facial cycles, by \cref{main:CoarseMenger:BoundedCycleSpace}, we obtain the coarse Menger property for planar graphs whose (internal) faces have bounded size.
	
	\begin{maincorollary} \label{maincor:CoarseMenger:planar}
		For every $\ell > 0$, the class of finite planar graphs whose internal faces have length at most~$\ell$ has the coarse Menger property. \qed
	\end{maincorollary}
	
	\noindent Moreover, \cref{main:CoarseMenger:BoundedCycleSpace} also implies that every (infinite) planar quasi-transitive graph has the coarse Menger property~\cite{HCycleSpacePlanarQT}.
	\medskip

	Recently, Liu~\cite{CHLCoarseMenger} proved that complete Riemannian planes have the weak coarse Menger property.
	As a further corollary of \cref{maincor:CoarseMenger:planar}, we obtain that complete Riemannian planes in fact have the (exact) coarse Menger property.
	This follows from a recent theorem of Davies~\cite{DStringGraphs} that complete Riemannian planes are quasi-isometric to planar triangulations, the fact that the coarse Menger property is preserved by quasi-isometry~\cite{CHLCoarseMenger}, and \cref{maincor:CoarseMenger:planar}.
	\medskip
	
	We remark that the assumption that the cycle space of a graph is generated by cycles of bounded length has already proved fruitful in the context of coarse graph theory for vertex-transitive graphs~\cite{AHAsymptoticGrids,M24+,EGG25,MMSkPlanar}. For example, MacManus~\cite{M24+} proved that a locally finite Cayley graph of a finitely presented group is planar if and only if it excludes some finite graph as an asymptotic minor.\footnote{See e.g.\ \cite{GP23} for the definitions of `quasi-isometric' and `asymptotic minor'.} Moreover, Albrechtsen and Hamann~\cite{AHAsymptoticGrids} showed that a quasi-transitive graph whose cycle space is generated by cycles of bounded length is quasi-isometric to a tree if and only if it does not contain the infinite grid on $\Z \times \Z$ as an asymptotic minor. In their~\cite{AHAsymptoticGrids} proof, they only use the assumption that the cycle space of the graph is generated by cycles of  length at most $\kappa \in \N$ to invoke their lemma \cite[Lemma~5.1]{AHAsymptoticGrids}, which asserts that for every connected subgraph $H$ of $G$, and every component~$C$ of $G-V(H)$, the $\kappatwo$-neighbourhood of~$H$ in~$C$ (that is $B_C\left(N_G(H) \cap V(C), \kappatwo\right)$ is connected (see \cref{lem:kappa/2NhoodIsConnected})). Our proof of \cref{main:CoarseMenger:BoundedCycleSpace} also relies on this lemma, and we will in fact only make use of the cycle space assumption to apply this lemma.
	\medskip

	This paper is organised as follows. In \cref{sec:Prelims} we recall some definitions and lemmas that we use throughout the paper. In \cref{sec:ProofSketch} we give a brief sketch of the proof, and provide some intuition why the assumption on the cycle space is useful to prove \cref{main:CoarseMenger:BoundedCycleSpace}. In \cref{sec:CoarseMenger:BoundedCycleSpace} we prove \cref{main:CoarseMenger:BoundedCycleSpace}.
	Finally, we conclude the paper with some concluding remarks in \cref{sec:ConcludingRemarks}.
	\medskip

	\section{Preliminaries} \label{sec:Prelims}
	
	All graphs in this paper are simple, and they may be infinite, unless otherwise stated.
	Our notions mainly follow~\cite{Bibel}. In what follows, we recall some definitions which we need later.
	
	For two sets $X,Y$ of vertices of $G$, an \defn{$X$--$Y$ path} meets~$X$ precisely in its first vertex and~$Y$ precisely in its last vertex. Moreover, if $P = p_0 \dots p_n$ is a path, then we denote by \defn{$p_iPp_j$} for $i, j \in \{0, \dots, n\}$ the subpath $p_i \dots p_j$ of~$P$, and we abbreviate $\defnm{p_iP} := p_iPp_n$ and $\defnm{Pp_i} := p_0Pp_i$.

	Given sets $U' \subseteq U$ of vertices of a graph $G$, a component $C$ of $G-U$ \defn{attaches} to~$U'$ if $C$ has a neighbour in~$U'$.
	The \defn{boundary $\partial_G X$} of a subgraph $X$ of $G$ is the set $N_G(V(G-X))$ of vertices of~$X$ that send in~$G$ an edge outside of~$X$. For example, the boundary $\partial_G C$ of a component $C$ of $G-U$ is $N_G(U) \cap V(C)$.

	\subsection{Distance and balls}
	
	Let $G$ be a graph.
	We write~\defn{$d_G(v, u)$} for the distance of the two vertices~$v$ and~$u$ in~$G$. 
	For two sets~$U$ and~$U'$ of vertices of~$G$, we write~\defn{$d_G(U, U')$} for the minimum distance of two elements of~$U$ and~$U'$, respectively.
	If one of~$U$ or~$U'$ is just a singleton, then we omit the braces, writing $d_G(v, U') := d_G(\{v\}, U')$ for $v \in V(G)$.
	If $X$ is a subgraph of $G$, then we abbreviate $d_G(U,V(X))$ as $d_G(U,X)$.
	
	Given a set~$U$ of vertices of~$G$, the \defn{ball (in~$G$) around~$U$ of radius $r \in \N$}, denoted by~\defn{$B_G(U, r)$}, is the set of all vertices in~$G$ of distance at most~$r$ from~$U$ in~$G$.
	If~$U = \{v\}$ for some~$v \in V(G)$, then we omit the braces, writing~$B_G(v, r)$ for the ball (in $G$) around~$v$ of radius~$r$.
	Additionally, we abbreviate the induced subgraph on $B_G(U,r)$ of $G$ with $\defnm{G[U,r]} := G[B_G(U,r)]$.
	If $X$ is a subgraph of $G$, then we abbreviate $B_G(V(X),r)$ and $G[V(X),r]$ as $B_G(X,r)$ and $G[X,r]$, respectively.

	\subsection{Cycle Space} \label{subsec:CycleSpace}
	
	Let $G$ be a graph. The \defn{edge space} of~$G$ is the vector space over the $2$-element field~$\mathbb{F}_2$ of all functions $E(G) \rightarrow \mathbb{F}_2$: its elements correspond to the subsets of $E(G)$ and vector addition corresponds to symmetric difference.
	The \defn{cycle space} of~$G$ is the subspace of the edge space of~$G$ spanned by all the cycles in~$G$~-- more precisely, by their edge sets. 
	
	To prove \cref{main:CoarseMenger:BoundedCycleSpace}, the only thing we need to know about the cycle space is the following lemma:

	\begin{lemma}[\rm{\cite[Lemma 5.1]{AHAsymptoticGrids}}]  \label{lem:kappa/2NhoodIsConnected}
		Let $G$ be a graph whose cycle space is generated by cycles of length at most $\kappa \in \N$, and let~$Y$ be a connected subgraph of $G$. Then for every component $C$ of $G-Y$ that attaches to $Y$, the graph $C\left[\partial_G C, \left\lfloor \frac{\kappa-2}{2}\right\rfloor\right]$ is connected.
	\end{lemma}

	\subsection{Coarse Menger theorem for two paths}
	
	In the proof of \cref{main:CoarseMenger:BoundedCycleSpace} we use that the coarse Menger conjecture (\cref{conj:CoarseMenger}) is true for two paths:
	
	\begin{theorem}[\rm{\cite[Theorem~2]{DistanceMengerForTwo} \& \cite[Theorem~8.1]{GP23}}] \label{thm:CoarseMengerFor2}
		For every $d \in \N$, for every graph $G$ and $X,Y \subseteq V(G)$, at least one of the following statements holds:
		\begin{enumerate}
			\item \label{itm:CoarseMengerFor2:Paths} There exist two $X$--$Y$ paths $P_1$, $P_2$ such that $d_G(P_1,P_2) \geq d$. 
			\item \label{itm:CoarseMengerFor2:Ball} There exists $z \in V(G)$ such that $B_G(z,129\cdot d)$ intersects every $X$--$Y$ path.
		\end{enumerate}
	\end{theorem}
	
	Moreover, we need the following remark about the two paths $P_1, P_2$ in~\ref{itm:CoarseMengerFor2:Paths} of~\cref{thm:CoarseMengerFor2}.
	
	\begin{proposition}[\rm{\cite[Corollary~8.14]{GP23}}] \label{prop:CoarseMengerFor2:EndPoints}
		In the setting of \cref{thm:CoarseMengerFor2}, if \ref{itm:CoarseMengerFor2:Paths} holds, then for every $x \in X$ and $y \in Y$ such that some $X$--$Y$ path has endvertices $x,y$, we can choose the paths $P_1, P_2$ in \ref{itm:CoarseMengerFor2:Paths} so that both $x, y$ are among their four endvertices.
	\end{proposition}

	In fact, we will later employ the following corollary of \cref{thm:CoarseMengerFor2} and \cref{prop:CoarseMengerFor2:EndPoints}.
	
	\begin{corollary} \label{cor:CoarseMengerFor2}
		For every $d \in \N$, for every graph $G$ and $X, Y \subseteq V(G)$, and for every $X$--$Y$ path~$q_0 \dots q_n$ in~$G$ at least one of the following statements holds:
		\begin{enumerate}
			\item \label{itm:cor:CoarseMengerFor2:Paths} There exist two $X$--$Y$ paths $P_1, P_2$ such that $d_G(P_1, P_2) \geq d$ and such that $q_0,q_n$ are among the four endvertices of~$P_1$ and $P_2$.
			\item \label{itm:cor:CoarseMengerFor2:Ball} There exists $i \in [n]$ such that $B := B_G(q_i, 258\cdot d)$ intersects every $X$--$Y$ path, and such that there are two $X$--$B$ paths $P_1, P_2$ with $d_G(P_1, P_2) \geq d$, and one of $P_1, P_2$ starts in $q_0$.
		\end{enumerate}
	\end{corollary}
	
	\begin{proof}
		Applying \cref{thm:CoarseMengerFor2} and \cref{prop:CoarseMengerFor2:EndPoints} yields either two $X$--$Y$ paths as in \ref{itm:cor:CoarseMengerFor2:Paths}, or some $z \in V(G)$ such that $B_G(z, 129\cdot d)$ intersects every $X$--$Y$ path. Since we are done in the former case, we may assume the latter. 
		Now first note that $B_G(z,129\cdot d)$ intersects the $X$--$Y$ path $q_0 \dots q_n$ in some vertex $q_i$, and thus $B_G(q_i, 258\cdot d) \supseteq B_G(z,129\cdot d)$ intersects every $X$--$Y$ path. We choose $i \in [n]$ minimal with this property, and claim that there are two $X$--$B_G(q_i, 258\cdot d)$ paths that are at distance at least $d$ from each other and such that one of them starts in $q_0$. Indeed, let $q_0 \dots q_j$ be the (unique) subpath of $q_0 \dots q_n$ that ends in $B := B_G(q_i, 258\cdot d)$, and that is internally disjoint from $B$. Applying \cref{thm:CoarseMengerFor2} and \cref{prop:CoarseMengerFor2:EndPoints} to $X$ and $Y := B$ and $d$ and to $x = q_0$ either yields the desired $X$--$B$ paths $P_1, P_2$, or it yields some $z' \in V(G)$ such that $B_G(z', 129\cdot d)$ intersects every $X$--$B$ path. In particular, it intersects $q_0 \dots q_j$ in some vertex $q_k$. But then $B_G(q_k, 258\cdot d)$ intersects every $X$--$B$ path, and hence every $X$--$Y$ path. As $k < i$, this contradicts the minimality of $i$.  
	\end{proof}

	\section{Sketch of proof} \label{sec:ProofSketch}

	All known proofs of the coarse Menger conjecture for $k=2$ \cite{DistanceMengerForTwo,GP23,CoarseMengerCounterex} roughly follow the same strategy. They start by taking a (shortest) $X$--$Y$ path~$P$, deleting balls around $P$ of some suitable radius $r >> d$, and then analysing the components of $G-B_G(P,r)$. If there is no set of small radius that separates $X$ from $Y$, then they find a set of components of $G-B_G(P,r)$ that are `well-behaved' (see e.g.\ \cite[Figure~4]{DistanceMengerForTwo}), and they conclude the proof by rerouting $P$ through some of these components, making space for a second path from $X$ to $Y$ at distance at least $d$ from $P$ (see e.g.\ \cite[Figure~6]{DistanceMengerForTwo} or \cite[Figure~6]{CoarseMengerCounterex}).
	
	Early hopes that \cref{conj:CoarseMenger} might be true for all $k \in \N$ have (partly) built on the idea that this proof strategy might generalise to larger $k$, using induction to find $k-1$ disjoint $X$--$Y$ paths $P_1, \dots, P_{k-1}$ pairwise at distance at least~$D >> 2r$. One obvious difficulty (amongst others) is that components of $G-\big(\bigcup_{i \in [k-1]} B_G(P_i, r)\big)$ might now attach to several $B_G(P_i, r)$'s, which makes analysing and rerouting the paths through them more challenging. 
	In the end, the counterexamples of Nguyen, Scott, and Seymour \cite{CoarseMengerCounterex,NewCounterexToCoarseMenger} showed that this idea was doomed to fail.
	
	However, if the cycle space of~$G$ is generated by cycles of bounded length, then the situation is simpler for two reasons (see \cref{fig:Hypertree}), paving the way for a proof by induction (where $P_1, \dots, P_{k-1}$ are $X$--$Y$ paths pairwise at distance at least~$D >> r >> d$, as given by the induction hypothesis):
	\smallskip
	
	[\cref{lem:Hypertree}]: First, for any two $i \neq j \in [k-1]$, there is at most one component of $G-\big(\bigcup_{n \in [k-1]} B_G(P_n,r)\big)$ that attaches to both $B_G(P_i,r)$ and $B_G(P_j,r)$. Moreover, the components of $G-\big(\bigcup_{n \in [k-1]} B_G(P_n,r)\big)$ connect the sets $B_G(P_i,r)$ in a `tree-like' way (see \cref{fig:Hypertree}).
	\smallskip
	
	[\cref{lem:kappa/2NhoodIsConnected}]: Second, for every component $C$ attaching to $B_G(P_i,r)$ and $B_G(P_j,r)$, say, both the $\kappatwo$-neighbourhood~$N_{i,C}$ of $B_G(P_i,r)$ in $C$ and the $\kappatwo$-neighbourhood~$N_{j,C}$ of $B_G(P_j,r)$ in~$C$ are connected, and at distance at least~$d$ from each other (indicated in grey in \cref{fig:Hypertree}). Therefore, we can reroute $P_i$ through~$N_{i,C}$, and at the same time reroute $P_j$ through $N_{j,C}$ without $P_i$ and $P_j$ getting too close.
	\begin{figure}[ht]
		\centering
		\includegraphics[width=0.67\linewidth]{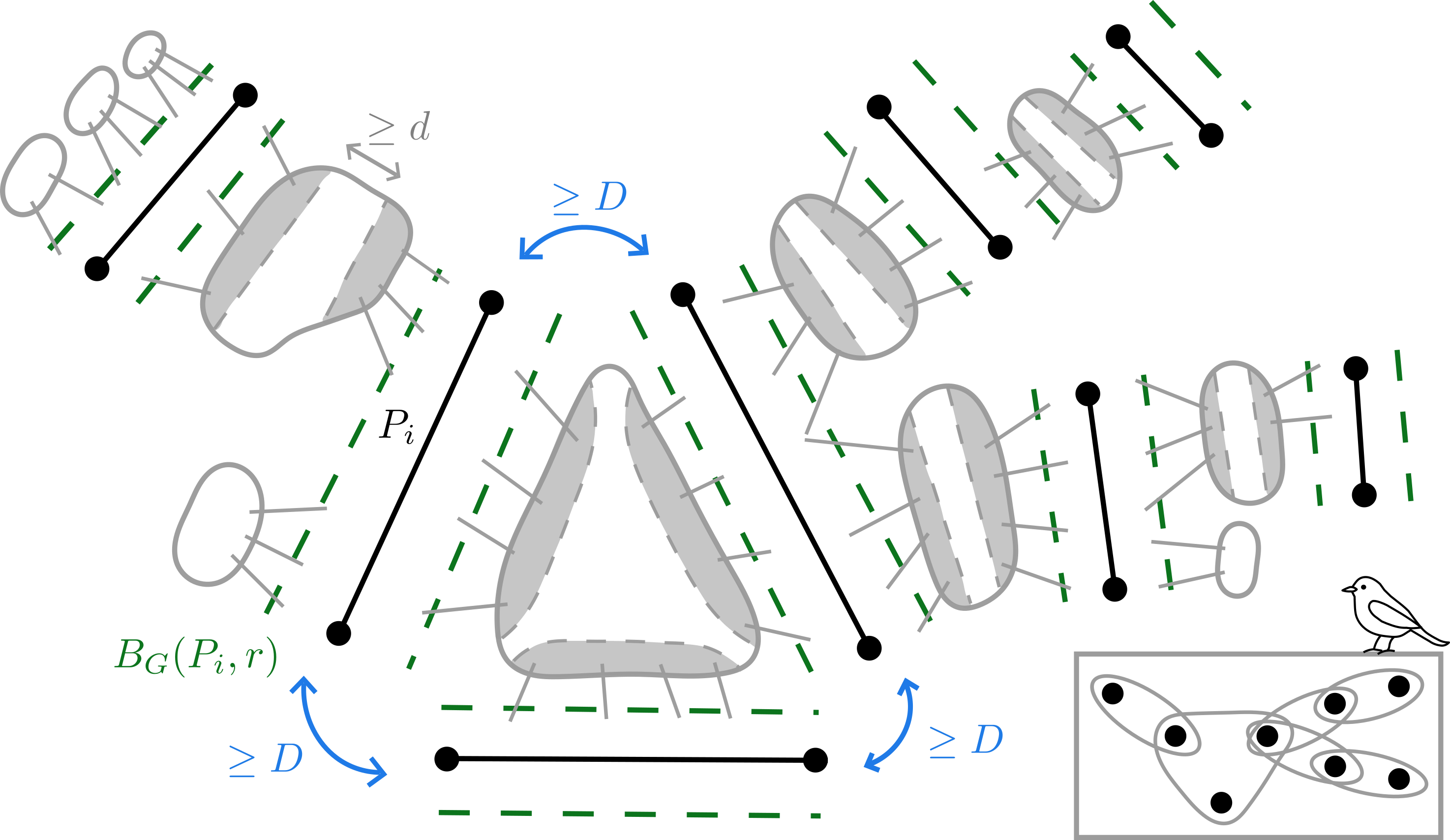}
		\vspace{0.5em}
		\caption{Depicted are paths $P_i$, pairwise at distance at least~$D$. The components~$C$ of $G-\big(\bigcup B_G(P_i,r)\big)$, for $r << D$, connect the sets $B_G(P_i,r)$ in a `tree-like' way: The hypergraph with vertex set $\{P_1, \dots, P_{k-1}\}$ in which the components~$C$ are turned into hyper-edges is acyclic (see right corner).\\
			Moreover, for every $P_i$ and every component~$C$ attaching to $B_G(P_i,r)$, the $\kappatwo$-neighbourhood of $B_G(P_i,r)$ in $C$ (indicated in grey) is connected, and at distance at least~$d$ from the other $\kappatwo$-neighbourhoods.}
		\label{fig:Hypertree}
	\end{figure}
	
	\smallskip

	[\cref{lem:2PathsInATube}]: We use \cref{thm:CoarseMengerFor2} and \cref{lem:kappa/2NhoodIsConnected} to show that, for every $i \in [k-1]$, there are either two $X$--$Y$~paths $W^i_1,W^i_2$ at distance at least~$d$ that avoid all $B_G(P_j,r)$ for $j \neq i$, or some small-radius ball~$B_i$ hits all $X$--$Y$ paths in $G$ whose endvertices are `close' to $P_i$ (see \cref{lem:2PathsInATube} and \cref{fig:2PathsInATube}). 
	
	If, for some $i \in [k-1]$, we find two such paths $W^i_1,W^i_2$, then $P_1, \dots, P_{i-1}, W^i_1, W^i_2, P_{i+1}, \dots, P_{k-1}$ are $k$~disjoint $X$--$Y$~paths, and they are pairwise at distance at least~$d$, thus concluding the proof in this case.
	
	Otherwise, we find for every $i \in [k-1]$ some small-radius ball~$B_i$ that hits all $X$--$Y$ paths in $G$ whose endvertices are `close' to $P_i$. By choosing the balls $B_i$ as `close' to~$X$ as possible, we can ensure that there are two $X$--$B_i$ paths $Q^i_1,Q^i_2$ at distance at least~$d$ that avoid all $B_G(P_j,r)$ with $j \neq i$ (see \cref{lem:2PathsInATube}~\ref{itm:2PathsInATube:2} and \cref{fig:2PathsInATube} (right)).
	\smallskip

	[\cref{subsec:ProofOfCoarseMenger:BoundedCycleSpace}]:  We choose the balls~$B_i$ from the application of \cref{lem:2PathsInATube} so that they are as `close' to~$Y$ as possible while still having two $X$--$B_i$ paths $Q^i_1,Q^i_2$, one of which is allowed to use one other set $B_G(P_j,r)$ (indicated in blue in \cref{fig:ProofSketch:FinalStep}). We then conclude the proof of \cref{main:CoarseMenger:BoundedCycleSpace} by showing that either these balls~$B_i$ hit all $X$--$Y$ paths, or we can find the desired $k$ disjoint $X$--$Y$ paths: If the balls $B_i$ do not hit all $X$--$Y$ paths, then there is a component~$C$ of $G-\big(\bigcup_{i \in [k-1]} B_G(P_i,r)\big)$ that attaches to some $B_G(P_i,r)$ `before'~$B_i$ and to some $B_G(P_j,r)$ `behind' $B_j$ (indicated in grey in \cref{fig:ProofSketch:FinalStep} (left)). By using \cref{lem:2PathsInATube}, this either yields a new ball $B'_j$ contradicting the optimality of $B_j$ (\cref{fig:ProofSketch:FinalStep} (left)), or we can reroute the paths $P_i$ to obtain $k$ disjoint $X$--$Y$ paths, pairwise at distance at least~$d$ (see  \cref{fig:ProofSketch:FinalStep} (right)). 
	\begin{figure}[ht]
		\begin{subfigure}[b]{0.5\linewidth}
			\centering
			\includegraphics[width=0.65\linewidth]{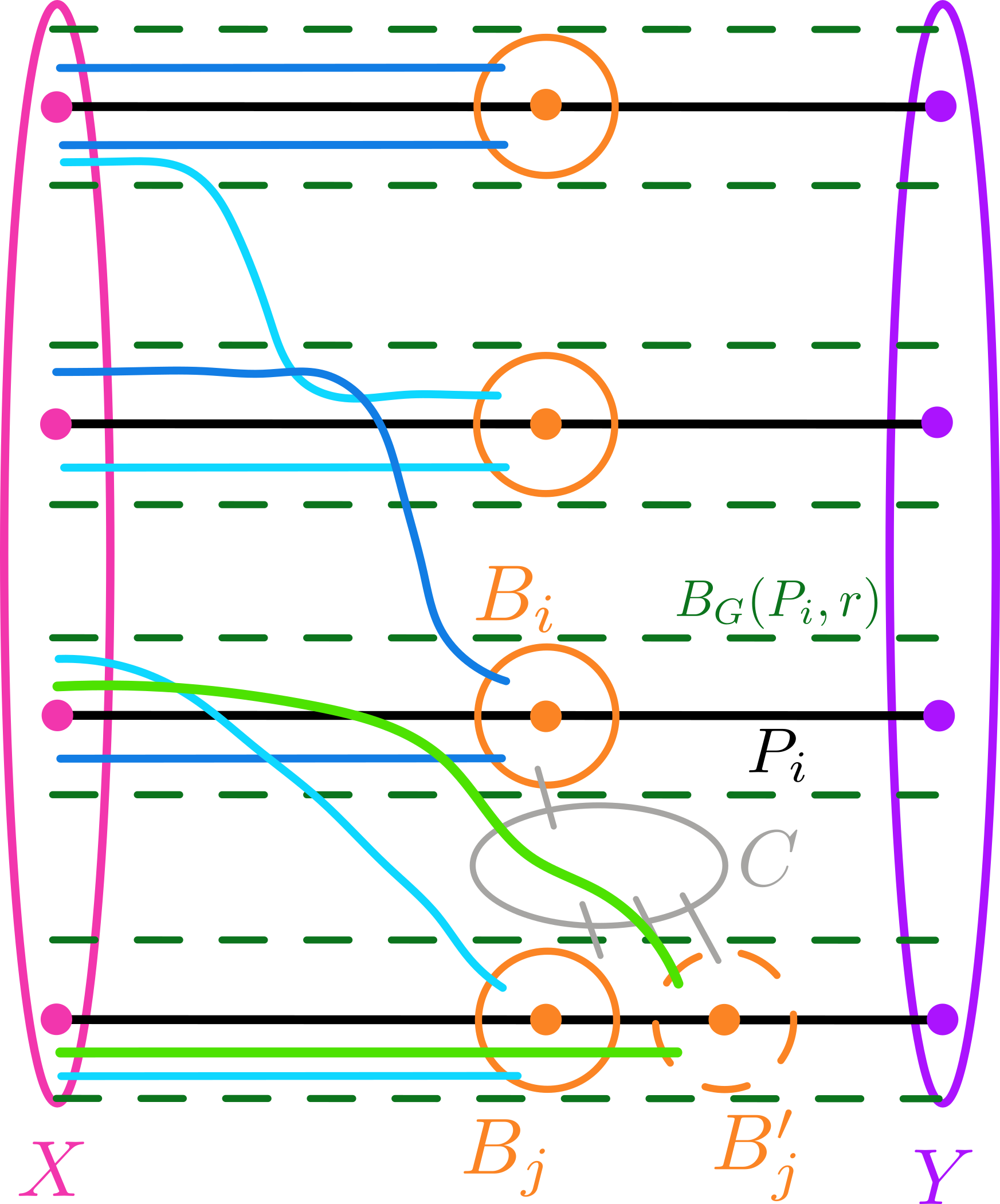}
		\end{subfigure}
		\begin{subfigure}[b]{0.5\linewidth}
			\centering
			\includegraphics[width=0.65\linewidth]{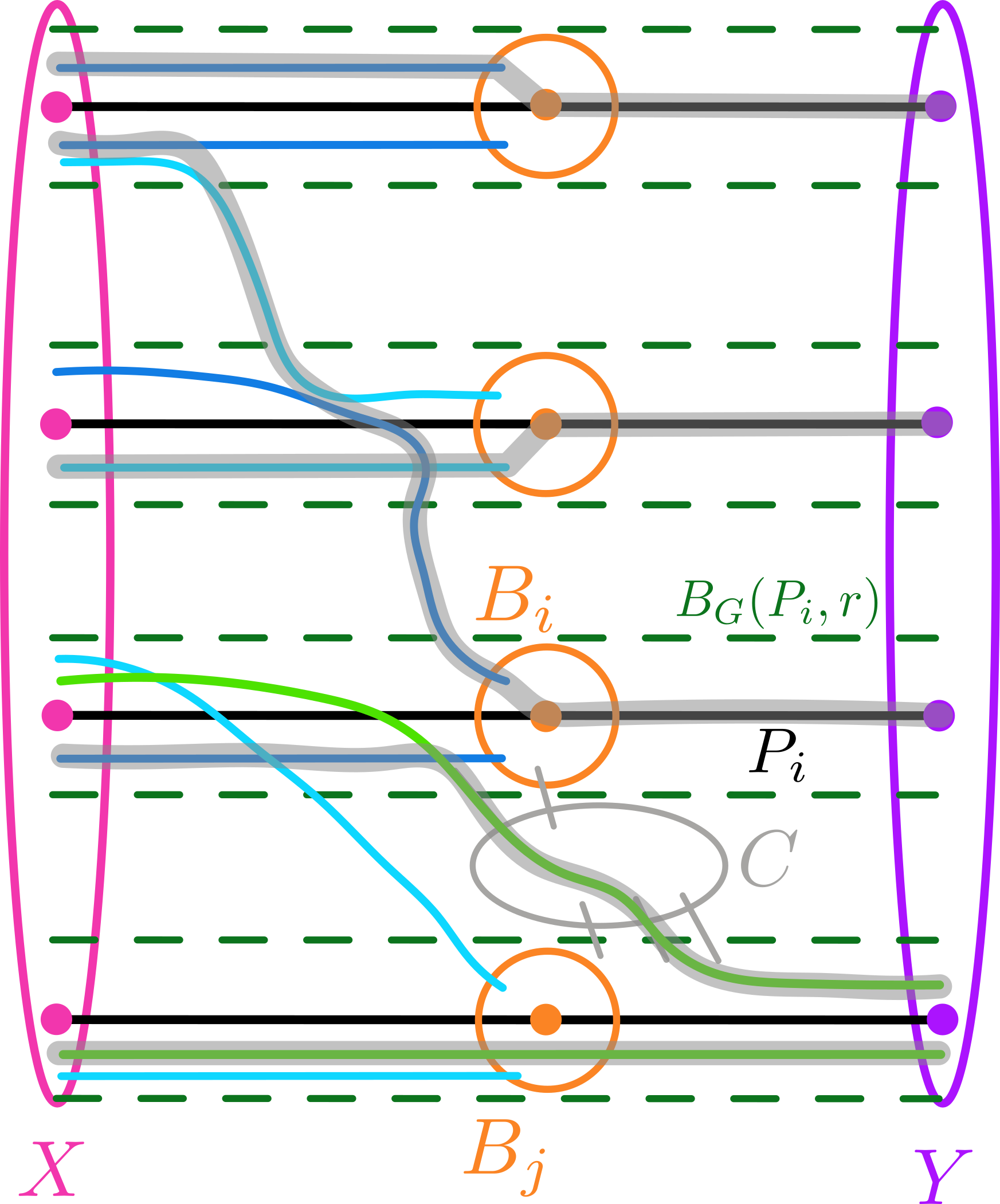}
		\end{subfigure}
		\vspace{0.2em}
		\caption{By the induction hypothesis, we find $k-1$ $X$--$Y$ paths $P_i$ that are pairwise at distance at least~$D$. \cref{lem:2PathsInATube} yields for every $P_i$ some small-radius ball $B_i$ (depicted in orange) and two $X$--$B_i$~paths (in blue) that are pairwise at distance at least $d$, one of which avoids all other $B_G(P_j,r)$. If the~$B_i$ do not hit all $X$--$Y$ paths, then there is a component~$C$ that attaches `before' some~$B_i$ and `behind' some~$B_j$ (in grey). Invoking \cref{lem:2PathsInATube} again yields either a `better' ball~$B'_j$ (left), or two $X$--$Y$ paths (right, in green), which we can use to find the desired $k$ $X$--$Y$ paths that are pairwise at distance at least~$d$ (right, in grey).}
		\label{fig:ProofSketch:FinalStep}
	\end{figure}

	\section{The proof} \label{sec:CoarseMenger:BoundedCycleSpace}
	
	In this section we prove \cref{main:CoarseMenger:BoundedCycleSpace}. 
	For this, we first show two preparatory lemmas in \cref{subsec:PreparatoyLemmas} and then prove \cref{main:CoarseMenger:BoundedCycleSpace} in \cref{subsec:ProofOfCoarseMenger:BoundedCycleSpace}.
	
	\subsection{Preparatory lemmas} \label{subsec:PreparatoyLemmas}
	
	We begin by showing that if the cycle space of a graph $G$ is generated by cycles of length at most~$\kappa$, and $U_1, \dots, U_n$ are connected vertex sets of $G$ that are pairwise at distance at least $\left\lfloor\frac{\kappa}{2}\right\rfloor$ from each other, then the components of $G-\bigcup_{i \in [n]} U_i$ connect the sets $U_i$ in a `tree-like' way (see \cref{fig:Hypertree} for an illustration, where the sets $B_G(P_i,r)$ play the role of the sets $U_i$).
	
	\begin{lemma} \label{lem:Hypertree}
		Let $n, r \in \N$, let $G$ be a graph whose cycle space is generated by cycles of length at most $\kappa \in \N$, and let $U_1, \dots, U_n$ be connected subsets of~$V(G)$ such that $d_G(U_i, U_j) \geq \left\lfloor \frac{\kappa}{2}\right\rfloor$ for all $i \neq j \in [n]$. 
		
		Let $H$ be the (multi-)hypergraph with vertex set $\{U_1, \dots, U_n\}$ and edge set $\big\{e_C : C \in \cC\big(G-(\bigcup_{i \in [n]} U_i)\big)\big\}$ where $e_C$ is the set of precisely those $U_i$, with $i \in [n]$, that send an edge to $C$.
		Then $H$ is acyclic. 
	\end{lemma}
	
	Note that acyclic here implies that $H$ is a hypertree, and that no two of its (hyper)edges intersect in more than one vertex.
	In particular, \cref{lem:Hypertree} implies that there are at most~$n-1$ components of $G-\big(\bigcup_{i \in [n]} U_i\big)$ that attach to more than one~$U_i$, and that for any two distinct $U_i \neq U_j$, there is at most one component of $G-\big(\bigcup_{i \in [n]} U_i\big)$ that attaches to $U_i$ and to $U_j$.
	
	\begin{proof}
		Suppose for a contradiction that $H$ contains a cycle\footnote{That means $D$ is an alternating sequence $(d_1e_1d_2, \dots, d_n,e_n,d_1)$ of distinct vertices $d_i$ and (hyper)edges $e_i$ such that $d_i$ is contained in $e_i$ and $e_{i+1}$.}~$D$.
		Without loss of generality, assume that $U_1 \in V(D)$, and let $C \neq C' \in \cC$ such that $e_C, e_{C'}$ are the two edges of~$D$ that are incident with~$U_1$. Let further~$U_i$ and~$U_j$ be the neighbours of~$U_1$ in~$D$, i.e.\ $U_i, U_j$ are incident with $e_C$ and $e_{C'}$, respectively (note that possibly $U_i = U_j$ if $e_C$ and $e_{C'}$ intersect in two vertices of $H$, in which case $D$ is a $2$-cycle). 
		Since $D$ is a cycle, $D-U_1$ is still connected, and hence 
		\[
		U := \bigcup_{U_i \in V(D) \setminus \{U_1\}} U_i \cup \bigcup_{e_{C''} \in E(D)} V(C'') \subseteq V(G)\setminus U_1
		\]
		is connected in $G$. 
		Thus, there exists a component $\widetilde{C}$ of $G-U_1$ such that $V(C), V(C') \subseteq U \subseteq V(\widetilde{C})$. 
		
		Let $v \in \partial_G C$ and $v' \in \partial_G C'$ send edges to $U_1$. Then $v, v' \in \partial_G \widetilde{C}$, and hence, by \cref{lem:kappa/2NhoodIsConnected}, there exists an $v$--$v'$ path $Q$ in $\widetilde{C}\left[\partial_G \widetilde{C}, \kappatwo\right]$. 
		Since $d_G(U_1, U_k) \geq \left\lfloor \frac{\kappa}{2}\right\rfloor$ for all $k \in \{2, \dots, n\}$, it follows that $Q$ is contained in some component $C''$ of $G - \big(\bigcup_{i \in [n]} U_i\big)$. As $Q$ starts in $C$ and ends in $C'$, this implies that $C = C'' =  C'$, which contradicts $C \neq C'$.
	\end{proof}
	
	The next lemma is essentially a variant of the coarse Menger theorem for two paths (more precisely of \cref{cor:CoarseMengerFor2}) for graphs whose cycle space is generated by cycles of bounded length: Given $X, Y \subseteq V(G)$ and an $X$--$Y$ path~$P$ in~$G$, if $X$ and $Y$ cannot be separated by a small-radius ball, then there are two $X$--$Y$~paths that are far apart from each other but at the same time `close' to the path $P$ (see \cref{fig:2PathsInATube}):

	\begin{lemma} \label{lem:2PathsInATube}
		Let $d,r \in \N$ with $r \geq 258\cdot d$, and let $G$ be a graph whose cycle space is generated by cycles of length at most $\kappa \in \N$. Let $X, Y,O \subseteq V(G)$, and let $P$ be an $X$--$Y$ path in~$G$ with $d_G(P,O) > r+\kappatwo$.
		Let~$\cC$ be the set of components of $G - \big(B_G(P,r) \cup O\big)$, and set $A' := A \cap \big(B_G(P,r) \bigcup_{C \in \cC}  V(C)\big)$ for $A \in \{X,Y\}$.
		Then at least one of the following statements holds:
		\begin{enumerate}[label=\rm{(\roman*)}]
			\item \label{itm:2PathsInATube:1} There are two $X'$--$Y'$ paths $Q_1, Q_2$ such that 
			\begin{enumerate}[label=\rm{(i.\alph*)}]
				\item \label{itm:2PathsInATube:1a} $d_G(Q_1, Q_2) \geq d$, and every component in $\cC$ is at distance at most~$d$ from at most one of $Q_1,Q_2$,
				\item \label{itm:2PathsInATube:1b} one of $Q_1,Q_2$ starts in $X \cap B_G(P,r)$, and one of $Q_1,Q_2$ ends in $Y \cap B_G(P,r)$, and
				\item \label{itm:2PathsInATube:1c} $Q_i$, for $i \in [2]$, is contained in $G\left[P, r+\kappatwo\right] \cup C_{x_i} \cup C_{y_i}$ where $C_{x_i}, C_{y_i} \in \cC$ are the components containing the endvertices $x_i,y_i$ of $Q_i$, respectively (where $C_{x_i},C_{y_i} := \emptyset$ if $x_i,y_i \in B_G(P,r)$).
			\end{enumerate} 
			\item \label{itm:2PathsInATube:2} There is a vertex $z \in V(P)$ such that $B:=B_G(z, 258\cdot d)$ intersects every $X'$--$Y'$ path and such that there are two $X'$--$B$ paths satisfying \ref{itm:2PathsInATube:1a} and \ref{itm:2PathsInATube:1c} such that one of them starts in $X\cap B_G(P,r)$.
		\end{enumerate}
	\end{lemma}
	
	\begin{figure}[ht]
		\centering
		\begin{subfigure}[b]{0.48\linewidth}
			\centering
			\includegraphics[scale=0.45]{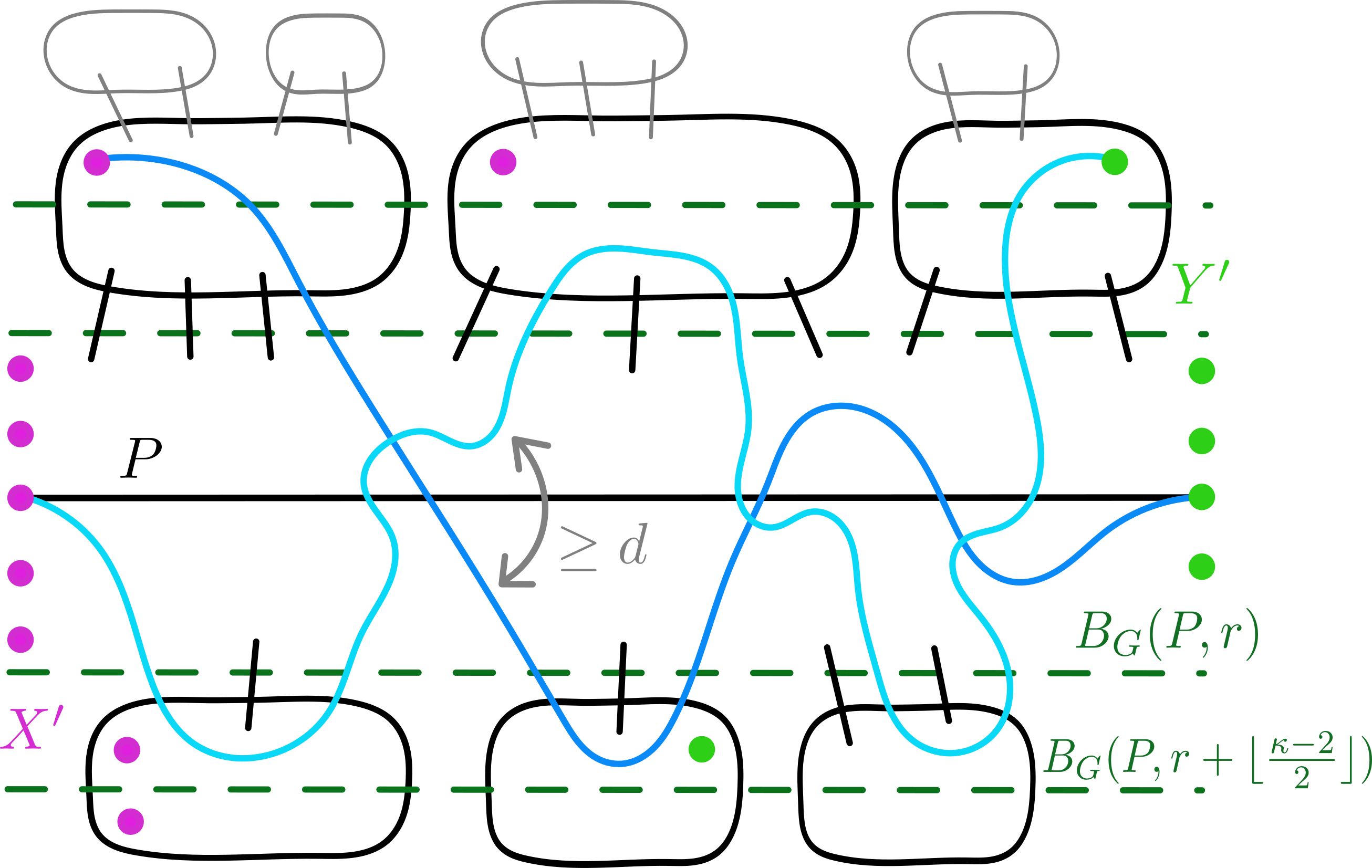}
		\end{subfigure}
		\begin{subfigure}[b]{0.48\linewidth}
			\centering
			\includegraphics[scale=0.45]{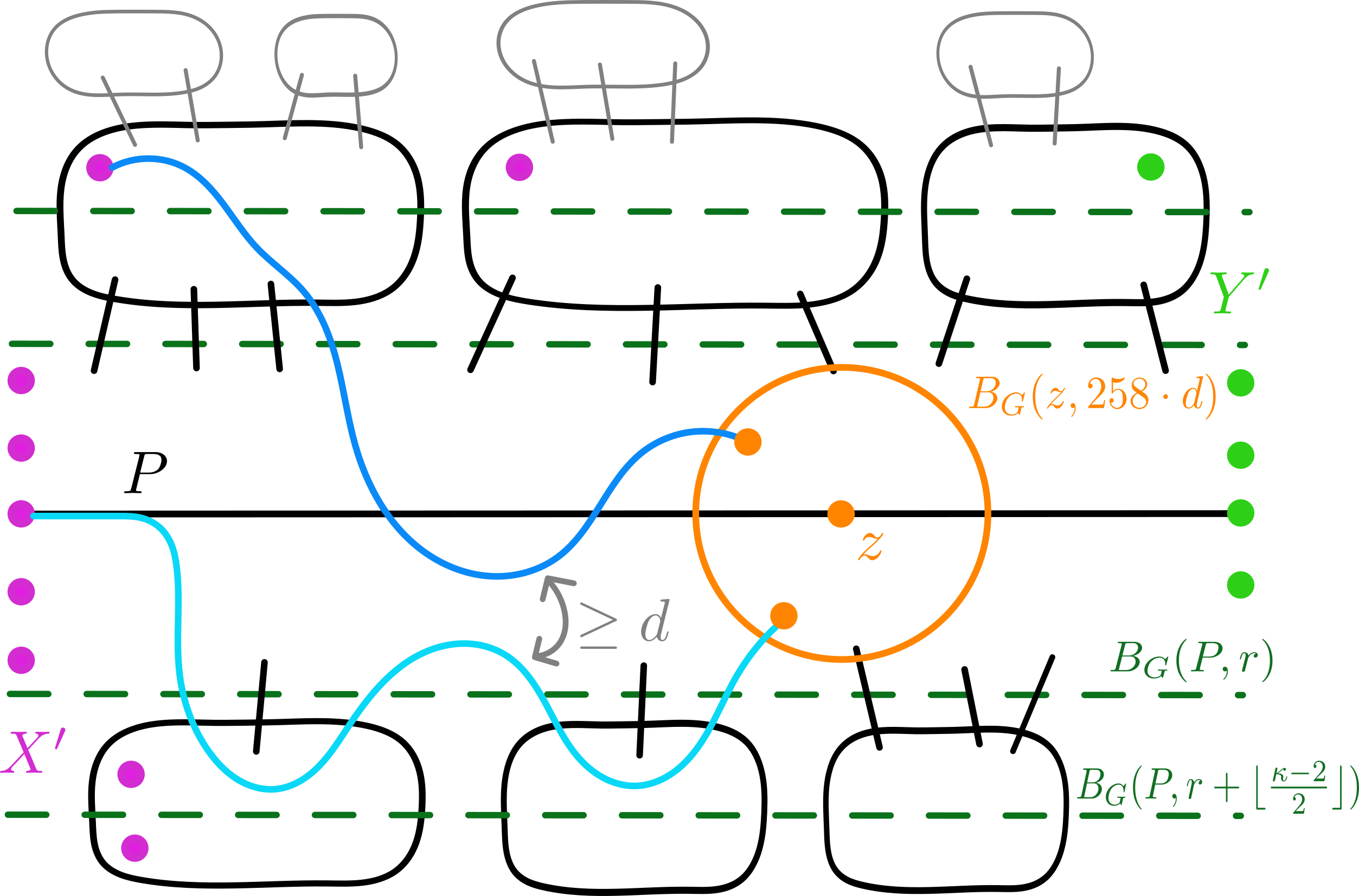}
		\end{subfigure}
		\vspace{0.5em}
		\caption{An illustration of \cref{lem:2PathsInATube}: There are either two $X'$--$Y'$ paths as in \ref{itm:2PathsInATube:1}, which are at distance at least~$d$ and which are (almost) contained in $B_G\left(P, r+\kappatwo\right)$, except for possibly some initial or final segment (left), or there is a vertex~$z \in V(P)$ such that $B_G(z,258d)$ intersects every $X'$--$Y'$ path and such that there are two $X'$--$B$ paths as in \ref{itm:2PathsInATube:2}, which are (almost) contained in $B_G\left(P, r+\kappatwo\right)$ (right).}
		\label{fig:2PathsInATube}
	\end{figure}

	\begin{proof}
		Let \defn{$G'$} be the graph obtained from $G$ by contracting every component $C$ of $G-B_G(P,r)$ to a vertex~$v_C$, and set $\defnm{A''} := \big(A \cap B_G(P,r)\big) \cup \big\{v_C : A' \cap V(C) \neq \emptyset\big\}$ for $A \in \{X,Y\}$.
		Applying \cref{cor:CoarseMengerFor2} in~$G'$ to $P =: p_0 \dots p_n$ and $X'',Y''$ yields either two $X''$--$Y''$ paths $Q'_1, Q'_2$ in $G'$ with $d_{G'}(Q'_1, Q'_2) \geq d$ such that $p_0,p_n$ are among the four endvertices of $Q'_1,Q'_2$, or a vertex $z \in V(P)$ such that $\defnm{B} := B_{G'}(z, 258\cdot d)$ intersects every $X''$--$Y''$ path in $G'$, and such that there are two $X''$--$B$ paths $W'_1, W'_2$ in $G'$ with $d_{G'}(W'_1, W'_2) \geq d$ such that $p_0$ is among the four endvertices of $W'_1,W'_2$. 
		\medskip
		
		Let us first assume the former. By \cref{lem:kappa/2NhoodIsConnected}, the $\kappatwo$-boundary $C\left[\partial_G C, \kappatwo\right]$ of every component~$C$ of~$G-U$ is connected. Therefore, for every $i \in [2]$ and vertex $v_C \in V(Q'_i)$ that is not an endvertex of $Q'_i$, we may replace $v_C$ in~$Q'_i$ by some $u$--$w$ path in $C\left[\partial_G C, \kappatwo\right]$ where $u, w \in \partial_G C$ are chosen so that they are incident with the edges of~$Q'_i$ that are incident with~$v_C$. 
		
		If an endvertex of $Q'_i$ is of the form~$v_C$, then let $w$ be a vertex in $X' \cap V(C)$ (or $Y' \cap V(C)$) and let $C_w \in \cC$ be the component of $G-\big(B_G(P,r) \cup O\big)$ containing~$w$.
		Since $C_w$ is contained in~$C$ and sends an edge to $B_G(P,r)$, \cref{lem:kappa/2NhoodIsConnected} implies that there is a path in $C\left[\partial_G C, \kappatwo\right]$ between $C_w$ and the edge~$e$ of~$Q'_i$ that is incident with~$v_C$. Since $C_w$ is a component of $G-\big(B_G(P,r) \cup O\big)$ and this path avoids both $B_G(P,r)$ and $O$, it follows that $e$ has an endvertex~$u$ in~$C_w$.
		Therefore, we may replace $v_C$ in~$Q'_i$ with an $w$--$u$ path in~$C_w$.
		This construction yields paths $Q_1, Q_2$ that are contained in~$G$. Moreover, they start and end in $X',Y'$, respectively, and they satisfy~\ref{itm:2PathsInATube:1a} as $d_G(Q_1, Q_2) \geq d_{G'}(Q'_1,Q'_2) \geq d$. Moreover, since $p_0,p_n$ were among the four endvertices of $Q'_1,Q'_2$, this is still true for $Q_1,Q_2$, and thus $Q_1,Q_2$ satisfy \ref{itm:2PathsInATube:1b}. Since by construction, $Q_1,Q_2$ satisfy \ref{itm:2PathsInATube:1c}, this implies that $Q_1,Q_2$ are paths as in~\ref{itm:2PathsInATube:1}.
		\medskip
		
		Let us now consider the former case. As $r \geq 258\cdot d$ and $z \in V(P)$, the ball $B = B_{G'}(z, 258\cdot d)$ does not contain any vertices of $G'$ of the form $v_C$, and hence $B = B_G(z,258\cdot d)$. As every $X'$--$Y'$ path in $G$ induces an $X''$--$Y''$ path in~$G'$, it follows that $B = B_G(z,258\cdot d)$ meets all $X'$--$Y'$ paths in~$G'$. By using the same construction as above, we can turn $W'_1, W'_2$ into $X'$--$B$ paths in~$G$ satisfying \ref{itm:2PathsInATube:1a} and \ref{itm:2PathsInATube:1c} such that one of them starts in~$p_0 \in X \cap B_G(P,r)$. Therefore, $z$ is as in~\ref{itm:2PathsInATube:2}.
	\end{proof}

	\subsection{Proof of \texorpdfstring{\cref{main:CoarseMenger:BoundedCycleSpace}}{Theorem 2}} 
	\label{subsec:ProofOfCoarseMenger:BoundedCycleSpace}
	
	We can now prove \cref{main:CoarseMenger:BoundedCycleSpace}, which we restate here for convenience.
	
	\begin{customthm}{\cref*{main:CoarseMenger:BoundedCycleSpace}} \label{thm:CoarseMenger:CycleSpace}
		There exists a function $c : \N^3 \rightarrow \N$ such that, for every $k, d \in \N$, for every graph~$G$ whose cycle space is generated by cycles of length at most $\kappa \in \N$, and for every $X, Y \subseteq V(G)$, at least one of the following statements holds:
		\begin{enumerate}[label=\rm{(\roman*)}]
			\item \label{itm:CoarseMenger:CycleSpace:Paths:Copy} There are $k$ disjoint $X$--$Y$ paths in $G$ which are pairwise at distance at least $d$ from each other.
			\item \label{itm:CoarseMenger:CycleSpace:Balls:Copy} There is a set $Z \subseteq V(G)$ of size at most $k-1$ such that $B_G(Z,c(\kappa, k,d))$ intersects every $X$--$Y$ path.
		\end{enumerate}
	\end{customthm}
	
	\begin{proof}
		Throughout the proof, we use parameters $r,d' \in \N$, which we define as $\defnm{r} := 1032\cdot d$ and $\defnm{d'} := 2r+d+\kappa+1$. 
		We prove the assertion with the function $c : \N^3 \to \N$ where $\defnm{c(\kappa,1,d)} := 0$ and $\defnm{c(\kappa, 2, d)} := 129\cdot d$ and $\defnm{c(\kappa, k+1,d)} := c(\kappa,k,d')$ for all $k \geq 2$. In particular, $c \in \mathcal{O}(2065^k \cdot(d+\kappa))$.
		\smallskip
		
		We proceed by induction on the number $k \in \N$ of desired paths. The base case $k = 1$ is trivial, and the case $k=2$ follows from \cref{thm:CoarseMengerFor2}, so let $k \geq 2$ be given, and assume that the assertion holds for $k$. 
		By the induction hypothesis, there is either a set $Z \subseteq V(G)$ of size at most~$k-1$ such that $B_G(Z, c(\kappa,k,d'))$ meets every $X$--$Y$ path, or there are $k$ disjoint $X$--$Y$ paths $P'_1, \dots, P'_k$ that are pairwise at distance at least~$d'$. In the former case, the set~$Z$ is as in~\cref{itm:CoarseMenger:CycleSpace:Balls:Copy} of \cref{thm:CoarseMenger:CycleSpace}, so we may assume the latter.
		\medskip
		
		We first need some definitions. 
		Let \defn{$\cC'$} be the set of components of $G-\big(\bigcup_{i \in [k]} B_G(P_i,r)\big)$.
		Further, let $\defnm{U_i}$ be the union over $B_G(P_i, r)$ together with all components $C \in \cC'$ that attach to $B_G(P_i,r)$.
		
		Assume that we are given a collection $\cB = \{B_1, \dots, B_k\}$ of subsets of~$V(G)$ such that each $B_i$ intersects every $(X \cap U_i)$--$(Y \cap U_i)$ path. For every $i \leq k$, let $\defnm{A_i^X(\cB)}$ and $\defnm{A_i^Y(\cB)}$ be the union over all components of $B_G(P_i, r) \setminus B_i$ that can be linked by a path that avoids $B_i \cup \bigcup_{j \neq i} B_G(P_j,r)$ to some vertex of~$X$ or~$Y$, respectively. Further, let $\defnm{A_i^r(\cB)} := B_G(P_i,r) \setminus (B'_i \cup A^X_i \cup A^Y_i)$. Note that $B_i, A_i^X(\cB), A_i^Y(\cB)$, and $A^r_i(\cB)$ are disjoint by the assumptions on $B_i$, and that $B_G(P_i,r) = B_i\, \dot\cup\, A_i^X(\cB)\, \dot\cup\, A_i^Y(\cB)\, \dot\cup\, A_i^r(\cB)$.    
		
		Let also $\defnm{\cC(\cB)}$ be the set of components of $G-\left(\bigcup_{i \in [k]} \big(B_G(P_i,r) \setminus A^r_i(\cB)\big)\right)$. 
		\medskip
		
		We now apply \cref{lem:2PathsInATube} to every path $P_i$ with $2d,r,X,Y$ and $O_i = \bigcup_{j \neq i} B_G(P_j,r)$. If, for some $i \in [k]$, this application yields two $(X \cap U_i)$--$(Y \cap U_i)$ paths $Q_1,Q_2$ as in~\ref{itm:2PathsInATube:1} of \cref{lem:2PathsInATube}, then $P_1, \dots, P_{i-1}, Q_1,Q_2, P_{i+1}, \dots, P_k$ are $k+1$ disjoint $X$--$Y$ paths as in~\ref{itm:CoarseMenger:CycleSpace:Paths:Copy} of \cref{thm:CoarseMenger:CycleSpace} (where $d_G(Q_1,Q_2) \geq d$ follows from~\ref{itm:2PathsInATube:1a} and $d_G(Q_n,P_j) \geq d$ follows from~\ref{itm:2PathsInATube:1c}). 
		Since we are done in this case, we may assume that \cref{lem:2PathsInATube} yields for every $i \in [k]$ outcome \ref{itm:2PathsInATube:2}: a vertex $z'_i \in V(P_i)$ such that $B_G(z'_i,258\cdot d)$ intersects every $(X \cap U_i)$--$(Y \cap U_i)$ path, and two $(X \cap U_i)$--$B_G(z'_i, 258\cdot d)$ paths $Q^i_1,Q^i_2$ satisfying \ref{itm:2PathsInATube:1a} and~\ref{itm:2PathsInATube:1c}. 
		
		Let $\defnm{Z} = \{\defnm{z_1}, \dots, \defnm{z_k}\}$ be a set of vertices $z_i \in V(P_i)$, for $i \in [k]$, such that (see \cref{fig:Bird}):
		\begin{enumerate}[label=\rm{(\arabic*)}]
			\item \label{itm:Def:vi:Ball} $\defnm{B_i} := B_G(z_i,516\cdot d)$ intersects every $(X \cap U_i)$--$(Y \cap U_i)$ path, 
			\item \label{itm:Def:vi:Paths} there are two $X$--$B_i$ paths $\defnm{Q^i_1}, \defnm{Q^i_2}$ such that $d_G(Q^i_1, Q^i_2) \geq 2d$ and such that
			\begin{enumerate}[label=\rm{(2\alph*)}]
				\item \label{itm:Def:vi:Q1} $Q^i_1 \subseteq G\left[P_i, r+\kappatwo\right]$, and
				\item \label{itm:Def:vi:Q2} there is $j_i \in [k]$ such that $V(Q^i_2) \subseteq \left(B_G\left(P_i \cup P_{j_i}, r+\kappatwo\right) \setminus B'_{j_i}\right) \cup \bigcup_{C \in \cC(\cB')} V(C)$ and $Q^i_2$ starts in $X \cap U_{j_i}$ where $\defnm{\cB'} := \{B'_1, \dots, B'_k\}$ and $\defnm{B'_n} := B_G(z_n, 1032\cdot d)$ for $n \in [k]$.
			\end{enumerate}
			\item \label{itm:Def:vi:HZ} $H(Z)$ (see below) is a forest without parallel edges but possibly with loops, and
			\item \label{itm:Def:vi:Minimise} among all sets of vertices that satisfy \ref{itm:Def:vi:Ball} to \ref{itm:Def:vi:HZ}, the set $Z$ minimises $\sum_{i \in [k]} ||b_iP_i||$,
		\end{enumerate}
		where \defn{$b_i$} is the last vertex on $P_i$ that is contained in~$B_i$ (when moving along $P$ from $X$ to $Y$).
		Moreover, $H(Z)$ is the auxiliary (multi)-graph $\defnm{H(Z)} := ([k],E_Z)$ with $E_Z := \{e_i = ij_i : i \in [k]\}$.
		
		Note that \ref{itm:Def:vi:Q2} is more general than \ref{itm:2PathsInATube:1c} of \cref{lem:2PathsInATube} as it allows $Q^i_2$ to intersect $B_G\left(P_{j_i}, r+\kappatwo\right)$ for one additional $j_i \in [k]\setminus \{i\}$.
		Note further that the vertices $z'_i$ from above satisfy \ref{itm:Def:vi:Ball} to \ref{itm:Def:vi:HZ} by \cref{itm:2PathsInATube:2} of \cref{lem:2PathsInATube} (where $H(Z)$ is the edgeless graph), and hence there exists a set~$Z$ satisfying \ref{itm:Def:vi:Ball} to \ref{itm:Def:vi:Minimise}.
		
		Abbreviate $\defnm{A_i^X} := A_i^X(\cB')$, $\defnm{A_i^Y} := A_i^Y(\cB')$, $\defnm{A_i^r} := A_i^r(\cB')$, and $\defnm{\cC} := \cC(\cB')$.
		\medskip
		
		\begin{figure}[ht]
			\centering
			\includegraphics[width=0.67\linewidth]{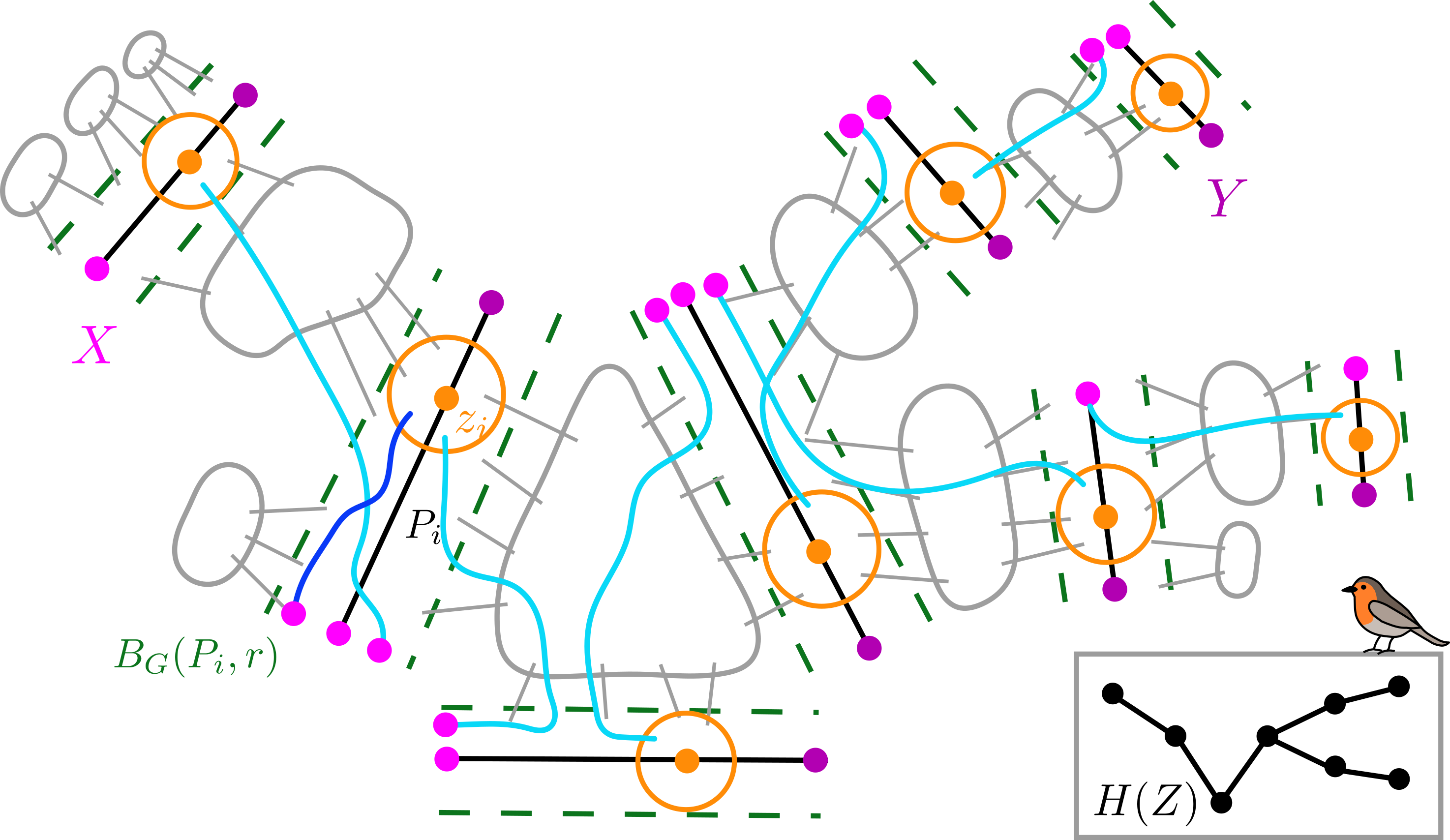}
			\vspace{0.5em}
			\caption{An illustration of the situation in the proof of \cref{thm:CoarseMenger:CycleSpace}. The paths $P_i$ are indicated in black, and the components of $G-\big(\bigcup_{i \in [k]} \big(B_G(P_i,r) \setminus A_i^r\big)\big)$ are indicated in grey. The balls $B_i$ around the vertices~$z_i$ are depicted in orange, and the paths $Q^i_2$ are indicated in light blue. Additionally, one path $Q^i_1$ is shown (in dark blue). The corresponding graph $H(Z)$ is a tree (see the lower right corner).}
			\label{fig:Bird}
		\end{figure}
		
		In the remainder of the proof we show that the balls $B'_i = B_G(z_i, 1032 \cdot d)$ around the vertices in~$Z$ hit all $X$--$Y$ paths, and hence $Z$ is as in~\ref{itm:CoarseMenger:CycleSpace:Balls:Copy} of \cref{thm:CoarseMenger:CycleSpace} (as $1032d \leq c(\kappa,k+1,d)$). 
		
		So suppose for a contradiction that the balls $B'_i$ do not hit all $X$--$Y$ paths. 
		
		\begin{claim} \label{claim:main:Path}
			There are $i,j \in [k]$ such that some component in $\cC$ attaches to $A^X_i$ and~$A^Y_j$.
		\end{claim}
		
		\begin{claimproof}
			By assumption, there is an $X$--$Y$ path $W$ that does not meet $\bigcup_{n \in [k]} B'_n$. We choose $W$ so that it intersects as few $B_G(P_n, r)$ as possible, and among those so that the number of its subpaths that start in some $B_G(P_i, r)$, end in some other $B_G(P_j,r)$, and are otherwise disjoint from $\bigcup_{n\in [k]} B_G(P_n,r)$ is as small as possible (that means $W$ (re-)enters each $B_G(P_n, r)$ after meeting some other $B_G(P_\ell,r)$ as few times as possible). If $W$ in fact does not intersect any $B_G(P_n, r)$, then we are done as $P_1, \dots, P_k, W$ would be a collection of paths as in~\ref{itm:CoarseMenger:CycleSpace:Paths:Copy} of \cref{thm:CoarseMenger:CycleSpace}. 
			If $W$ intersects precisely one $B_G(P_n, r)$, then $W$ starts in $X \cap U_n$ and ends in $Y \cap U_n$, which contradicts property~\ref{itm:Def:vi:Ball} of~$z_n$.
			
			Hence, $W$ meets at least two distinct sets $B_G(P_n, r)$. Let $i,j \in [k]$ be such that $B_G(P_i, r)$ is the first such set that $W$ meets when moving along $W$ from $X$ to $Y$, and $B_G(P_j,r)$ the last. 
			Since $W$ starts in~$X$ and ends in~$Y$, it meets $A^X_i$ and $A^Y_j$. 
			By the choice of $W$, it follows easily that $W$ avoids $B_G(P_n,r) \setminus A_n^r$ for all $n \neq i,j$ (because if $W$ meets some other $B_G(P_n,r) \setminus A_n^r$, then we could reroute $W$ so that it does not meet $B_G(P_i,r)$ before meeting $A^X_n$, or so that it does not meet $B_G(P_j,r)$ after meeting $A^Y_n$).
			
			It follows that $W$ contains some subpath between $B_G(P_i, r)$ and $B_G(P_j, r)$ that is internally disjoint from $B_G(P_i \cup P_j, r)\setminus (A_i^r \cup A_j^r)$ and which starts in $A^X_i$ and ends in $B_G(P_j,r)$, and this subpath must be contained in some component $C \in \cC$. 
			(Recall that $\cC$ is the set of components of $G-\big(\bigcup_{n \in [k]} \big(B_G(P_n, r) \setminus A_n^r\big)\big)$). In particular, $C$ sends an edge to $A^X_i$ and to $B_G(P_j,r) \setminus A_j^r$.
			Analogously, there is a component $C' \in \cC$ that sends an edge to $A^Y_j$ and to $B_G(P_i,r) \setminus A_i^r$. By applying \cref{lem:Hypertree} to the sets $B_G(P_n,r) \setminus A_n^r$, $n \in [k]$, (which are connected), it follows that $C = C'$, and thus $C$ is as desired.
		\end{claimproof}
		
		By \cref{claim:main:Path}, there are $i,j \in [k]$, without loss of generality say $i=1$ and $j=2$, and a component~\defn{$C_{1,2}$} of $G-\big(\bigcup_{n\in[k]} \big(B_G(P_n,r)\setminus A_n^r\big)\big)$ that attaches to $A^X_1$ and to $A^Y_2$. Pick some vertex~$\defnm{c}$ inside the component~\defn{$C'_{1,2}$} of $C_{1,2}-B_G(P_2,r) = C_{1,2}-A_2^r$ that attaches to $A_1^X$ (note that $C'_{1,2}$ is unique by \cref{lem:Hypertree}), and let $\defnm{X'} := X \cup \{c\}$. 
		
		Applying \cref{lem:2PathsInATube} in $G$ with $2d$, $X = X'$, $Y$, $P = P_2$ and $O = \bigcup_{n \in [k]\setminus\{2\}} \big(B_G(P_n, r) \setminus A_n^r\big)$ either yields two $X'$--$Y$ paths that satisfy \ref{itm:2PathsInATube:1a} to \ref{itm:2PathsInATube:1c} of \cref{lem:2PathsInATube}, or a vertex $z \in V(P_2)$ that satisfies \ref{itm:2PathsInATube:2} of \cref{lem:2PathsInATube}.
		Thus, we may conclude the proof by considering the following two cases:
		\medskip

		\noindent \textbf{Case 1:} \emph{There is a vertex $\defnm{z} \in V(P_2)$ such that $\defnm{B''_2} := B_G(z, 516\cdot d)$ intersects every $(X'\cap U_2)$--$(Y \cap U_2)$ path and such that there are two $(X'\cap U_2)$--$B''_2$ paths $W_1,W_2$ satisfying \ref{itm:2PathsInATube:1a} (with $2d$) and \ref{itm:2PathsInATube:1c} of \cref{lem:2PathsInATube} such that one of $W_1,W_2$ starts in $X' \cap B_G(P_2,r)$.}
		
		We show that $B''_2$ contains a vertex of~$P_2$ that appears on~$P_2$ after~$b_2$ (when moving along~$P_2$ from~$X$ to~$Y$), and that $\{z_1, z, z_3, \dots, z_k\}$ satisfies \ref{itm:Def:vi:Ball} to \ref{itm:Def:vi:HZ}. This then contradicts the choice of~$Z$ as $Z'$ would have been a better choice for~$Z$ because of \ref{itm:Def:vi:Minimise}, and thus concludes the proof in Case~1 (see \cref{fig:ProofSketch:FinalStep} (left)).
		\smallskip
		
		We first remark that 
		\labtequtag{eq:B2}
		{\emph{$B''_2$ avoids $B_2$, and there exists a $B''_2$--$(Y \cap U_2)$ path~$W$ in $G$ that avoids $B_2$.}}
		{\ding{100}}
		Indeed, if $B''_2$ meets $B_2$, then $B''_2 \subseteq B'_2$, which implies that $B'_2$ separates $c$ from~$Y \cap U_2$. But this contradicts that $C_{1,2} \ni c$ attaches to $A^Y_2$. Moreover, since $C_{1,2}$ attaches to $A^Y_2$, there is a $c$--$(Y \cap U_2)$ path $W'$ that avoids $B_2$. As $B''_2$ separates $c \in X' \cap U_2$ from $Y \cap U_2$, it follows that $B''_2$ meets $W'$. Hence, there is a subpath~$W$ of $W'$ with endvertices in $B''_2$ and $Y \cap U_2$. Since $W'$ avoids $B_2$, also $W$ avoids $B_2$.
		\medskip

		We now show that $B''_2$ contains a vertex of~$P_2$ that appears on~$P_2$ after~$b_2$. 
		By \eqref{eq:B2}, $B''_2$ avoids $B_2$. Let $b'_2$ be the first vertex on $P_2$ that is contained in $B_2$. Then $P_2b'_2 \cup G[B_2] \cup b_2P_2$ is connected and links $X \cap U_2$ to $Y \cap U_2$. As $B''_2$ avoids $B_2$ but separates $X \cap U_2$ from $Y \cap U_2$, it follows that $B''_2$ meets either $P_2b'_2$ or $b_2P_2$. Since we are done in the latter case, we may assume the former. 
		
		Let $b''_2 \in V(P_2b'_2) \cap B''_2$. 
		By \eqref{eq:B2}, there is a $B''_2$--$(Y \cap U_2)$ path $W$ that avoids $B_2$. Then $P_2b''_2 \cup G[B''_2] \cup W$ is connected, links $X \cap U_2$ to $Y \cap U_2$ and avoids $B_2$. This contradicts property~\ref{itm:Def:vi:Ball} of~$B_2$.
		\medskip
		
		Finally, we show that $\defnm{Z'} = \{z_1, z, z_3, \dots, z_k\}$ satisfies \ref{itm:Def:vi:Ball} to \ref{itm:Def:vi:HZ}, which concludes the proof in Case~1.
		
		\ref{itm:Def:vi:Ball}: By the assumption of Case~1, $B''_2$ satisfies~\ref{itm:Def:vi:Ball}. Hence, by~\ref{itm:Def:vi:Ball} of~$Z$, the set~$Z'$ satisfies \ref{itm:Def:vi:Ball}.
		
		\ref{itm:Def:vi:Paths}: By the assumption of Case~1, at least one of $W_1, W_2$ starts in $X' \cap B_G(P_2,r)$; without loss of generality assume that this path is $W_1$. In particular, $W_1$ starts in $X$. Let $\defnm{\tilde{Q}^2_1} := W_1$. If $W_2$ also starts in~$X$, then we set $\defnm{\tilde{Q}^2_2} := W_2$. In this case, $\tilde{Q}^2_1, \tilde{Q}^2_2$ satisfy \ref{itm:Def:vi:Paths} with $j_2 : =2$.
		
		Otherwise, if $W_2$ starts in~$c \notin X$, then by the choice of~$c \in V(C'_{1,2})$, there is a $c$--$(X \cap U_1)$ path $W'$ in $V(C'_{1,2}) \cup \big(B_G(P_1,r) \setminus B'_1\big) \cup \big(\bigcup_{C \in \cC'} V(C)\big)$. Let \defn{$\tilde{Q}^2_2$} be any $(X \cap U_1)$--$B''_2$ path in $W' \cup W_2$. 
		Then by definition, $\tilde{Q}^2_1$ and~$\tilde{Q}^2_2$ satisfy~\ref{itm:Def:vi:Paths} with $j_2 := 1$, where we note that $d_G(\tilde{Q}^i_1,\tilde{Q}^i_2) \geq 2d$ because $d_G(W_1,W_2) \geq 2d$ and $C'_{1,2}$ is distance at least~$2d$ from $W_1$ (as $W_1,W_2$ satisfy~\ref{itm:2PathsInATube:1a}).

		Moreover, for every $i \neq 2$, the paths $Q^i_1,Q^i_2$ still satisfy \ref{itm:Def:vi:Paths}. Indeed, it suffices to verify \ref{itm:Def:vi:Q2}, for which we only need to check that the paths~$Q^i_2$ with $j_i = 2$ avoid $B''_2$. So suppose for a contradiction that some such path $Q^i_2$ meets $B''_2$. By \eqref{eq:B2} and (the `old')~\ref{itm:Def:vi:Paths} of~$Q^i_2$, the subgraph $Q^i_2 \cup G[B''_2] \cup W$ links $X \cap U_2$ to $Y \cap U_2$ and avoids $B_2$. But this contradicts~\ref{itm:Def:vi:Ball} of~$B_2$.
		
		\ref{itm:Def:vi:HZ}: If $j_2 = 2$, then we only added a loop, so $H(Z')$ still satisfies \ref{itm:Def:vi:HZ}.
		Therefore, we may assume $j_2 = 1$.
		
		Suppose for a contradiction that $H(Z')$ contains a cycle $D$ (where possibly $D$ is a $2$-cycle consisting of two parallel edges). If $21 \notin E(D)$, then $D \subseteq H(Z)$ (as $E_{Z'} \setminus E_Z = \{21\}$) and thus $H(Z)$ did not satisfy~\ref{itm:Def:vi:HZ}, a contradiction. Therefore, $21 \in E(D)$.
		If also $12 \in E(D)$, then by \ref{itm:Def:vi:Q2} there exists an $(X \cap U_2)$--$C_{1,2}$ path that avoids $B_2$. But since $C_{1,2}$ attaches to $A^Y_2$, there is also a $C_{1,2}$--$(Y \cap U_2)$ path that avoids~$B_2$. Since $C_{1,2}$ is connected and avoids $B_2$, this contradicts that $B_2$ satisfies~\ref{itm:Def:vi:Ball}.
		
		Hence, we may assume, without loss of generality, that $32 \in E(D)$. By \ref{itm:Def:vi:Q2}, there is a component~$C_{2,3}$ of $G-\big(\bigcup_{i \in [k]} \big(B_G(P_i, r) \setminus A_i^r\big)\big)$ that attaches to $B_G(P_2,r) \setminus A_2^r$ and $B_G(P_3,r)\setminus A_3^r$. (Note that $C_{2,3}$ is unique by \cref{lem:Hypertree}).
		Since $32 \in E_Z$, there exists by \ref{itm:Def:vi:Q2} an $(X \cap U_2)$--$C_{2,3}$ path that avoids $B_2$. As $C_{2,3} = C_{1,2}$ by \cref{lem:Hypertree} (applied to the (connected) sets $B_G(P_n,r) \setminus A_n^r$) and because $D$ is a cycle, this yields the same contradiction as in the previous paragraph. 
		\medskip
		
		\noindent \textbf{Case 2:} \emph{There are two $(X' \cap U_2)$--$(Y\cap U_2)$ paths $W_1,W_2$ that satisfy~\ref{itm:2PathsInATube:1a} to~\ref{itm:2PathsInATube:1c} of \cref{lem:2PathsInATube} (with~$2d$).}
		
		We show that $G$ contains $k+1$ disjoint $X$--$Y$ paths that are pairwise at distance at least~$d$, and hence $G$ satisfies~\ref{itm:CoarseMenger:CycleSpace:Paths:Copy} (see~\cref{fig:ProofSketch:FinalStep} (right)).
		\smallskip
		
		For this, let us first note that the auxiliary graph \defn{$H$} with vertex set~$[k]$ and edge set $\{e_i = ij_i : i \in [k]\}$ where $j_2 = 1$ and every other~$j_i$ is as in~\ref{itm:Def:vi:Paths} is a forest (without parallel edges but possibly with loops). Indeed, $H = H(Z')$ where $Z'$ is as in Case~1, and thus the proof that $H$ is a forest is analogous to the proof of~\ref{itm:Def:vi:HZ} in Case~1 (where we remark that we never made use of $B''_2$ or any other assumption of Case~1). 
		\smallskip
		
		We now start with the construction of the $k+1$ $X$--$Y$ paths. 
		For this, let us first note that, since $H$ is a forest (without parallel edges but with loops), the sequence $\cS = 2,1,j_1, j_{j_1}, j_{j_{j_1}}, \dots$ is a path, which eventually ends in a leaf of~$H$. (More precisely, it eventually loops at a leaf).
		Note that by \ref{itm:Def:vi:Q2} there exists for every consecutive $m,n \in \cS$ a (unique) component \defn{$C'_{m}$} of $G-\big(B_G(P_m,r) \cup \bigcup_{i \neq m} \big(B_G(P_i,r) \setminus A^r_i\big)\big)$ that meets $Q^m_2$ and attaches to $B_G(P_n,r)$.
		Let $\defnm{\ell_1}, \defnm{\ell_2}, \dots, \defnm{\ell_n}$ be a subsequence of~$\cS$ where $\ell_1 = 2$, and $\ell_n$ is a leaf of~$H$, and $\ell_{i+1}$ is the last element of $\cS$ such that $C_{\ell_{i}}$ meets $Q^{\ell_{i+1}}_1 \cup Q^{\ell_{i+1}}_2$ if such an $\ell_{i+1}$ exist, otherwise we let $\ell_{i+1}$ be the successor of $\ell_i$ in $\cS$. 
		\medskip
		
		\begin{figure}[ht]
			\centering
			\begin{subfigure}[b]{0.48\textwidth}
				\centering
				\includegraphics[width=0.85\linewidth]{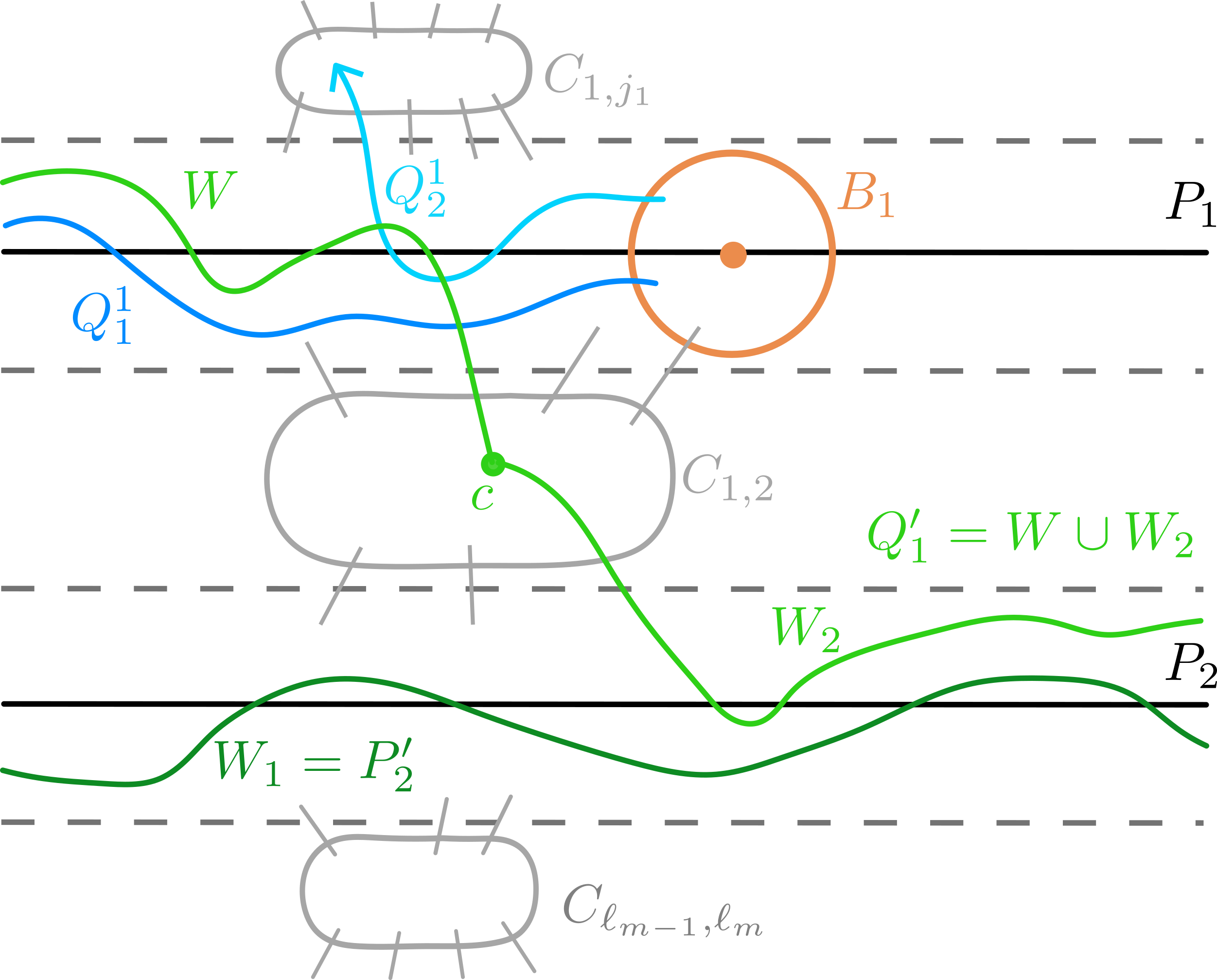}
			\end{subfigure}
			\begin{subfigure}[b]{0.48\textwidth}
				\centering
				\includegraphics[width=0.85\linewidth]{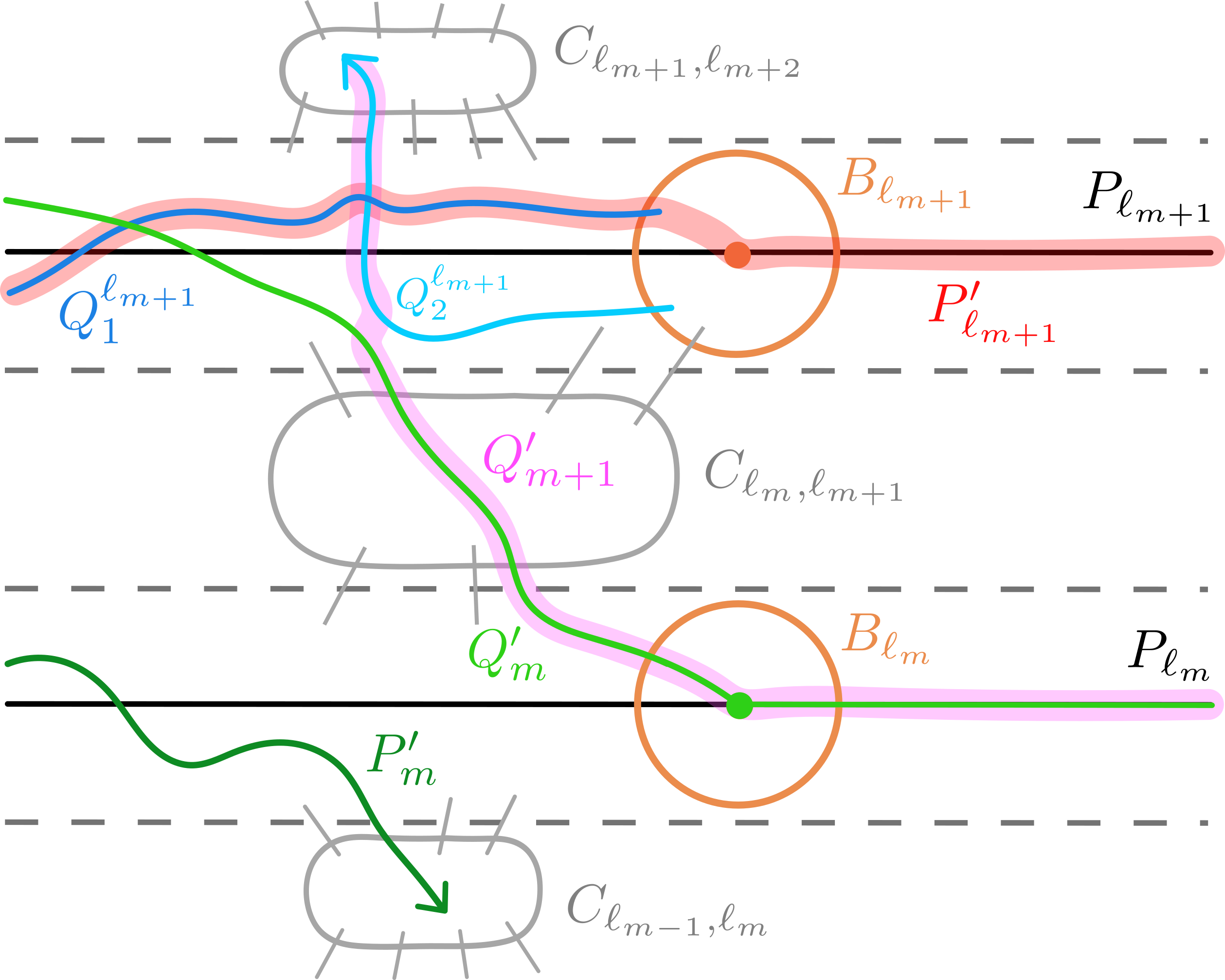}
			\end{subfigure}
			\vspace{0.5em}
			\caption{Illustration of the paths $P'_2, Q'_1$ (left) and $P'_{\ell_{m+1}}, Q'_{m+1}$ (right). On the right side, $i=2$.}
			\label{fig:Rerouting}
		\end{figure}
		
		\textbf{Construction of $\mathbf{P'_2 = P'_{\ell_1}}$ and $\mathbf{Q'_1}$:} (See \cref{fig:Rerouting} (left).) Without loss of generality assume that $W_1$ does not start in~$c$ (and hence it starts in $X$). 
		Set $\defnm{P'_2} := W_1$. If $W_2$ starts in $X$, too, then $P'_2,W_2$ are $X$--$Y$ paths, and they avoid $\bigcup_{i \neq 2} B_G(P_i,r)$ by~\ref{itm:2PathsInATube:1c}. Hence, $P_1,W_1,W_2,P_3, \dots, P_k$ is a collection of $k+1$ disjoint $X$--$Y$ paths as in~\ref{itm:CoarseMenger:CycleSpace:Paths:Copy} (where $d_G(W_1,W_2) \geq d$ holds by~\ref{itm:2PathsInATube:1a}, and $d_G(W_j,P_i) \geq r \geq d$ holds because $W_1,W_2$ avoid $\bigcup_{i \neq 2} B_G(P_i,r)$). 
		In this case, we terminate the construction and conclude the proof.
		
		Otherwise, $W_2$ starts in~$c$. Then there exists a $(X \cap U_{\ell_2})$--$c$ path~$W$ in $V(C'_{\ell_1}) \cup \big(B_G(P_{\ell_2},r)\setminus B'_{\ell_2}\big) \cup \big(\bigcup_{C \in \cC'} V(C)\big)$ (either because $C'_{\ell_1}$ intersects $Q^{\ell_1}_1 \cup Q^{\ell_2}_2$ or because $\ell_1 = 1$ and then $C'_{\ell_1} = C'_{1,2}$ and so $C'_{\ell_1}$ attaches to $A^X_1$). Let $\defnm{Q'_{1}}$ be any $(X \cap U_{\ell_2})$--$B''_2$ path in $W \cup W_2$. Note that $Q'_1$ avoids $B'_1$. 
		
		Since $P'_2$ avoids $\bigcup_{i \neq 2} B_G(P_i,r)$, the paths $P_1, P'_2,P_3,\dots, P_k$ form a collection of $k$ disjoint $X$--$Y$ paths that are pairwise at distance at least~$d$. 
		Moreover, $d_G(P'_2,Q'_{\ell_1}) \geq 2d$ because $d_G(W_1,W_2 \cup C'_{1,2}) \geq 2d$ by property \ref{itm:2PathsInATube:1a} of $W_1,W_2$. In particular, $d_G(W_1, C'_{\ell_1}) \geq 2d$ implies $d_G(P'_2, Q^{\ell_2}_1 \cup Q^{\ell_2}_2) \geq 2d$.
		\medskip
		
		Now assume that for some $m < n$ we have defined a collection of $k$ disjoint $X$--$Y$ paths $\{P'_{\ell_i} : i \leq m\} \cup \{P_j : j \in [k]\setminus \{\ell_1, \dots, \ell_m\}\}$ that are pairwise at distance at least~$d$, and a path~$Q'_m$ that starts in $X \cap U_{\ell_{m+1}}$, ends in~$Y$ and avoids $B'_{\ell_{m+1}}$ such that each $P'_i$ is at distance at least $2d$ from $B_G(P_j,r)$ and $Q^j_1,Q^j_2$ for every $j \in [k]\setminus \{\ell_1, \ldots, \ell_m\}$, and the path $Q'_m$ is at distance at least~$2d$ from all $P'_i$, and avoids all $B_G(P_j,r)\setminus A_j^r$ with $j \in [k] \setminus \{\ell_1, \ldots, \ell_{m+1}\}$.
		\medskip

		\textbf{Construction of $\mathbf{P'_{\ell_{m+1}}, Q'_{m+1}}$:}
		(See \cref{fig:Rerouting} (right).)
		Set $\ell := \ell_{m+1}$.
		If $Q'_m$ is at distance at least $d$ from $Q^\ell_1$, then we let $\defnm{P'_\ell}$ be an $X$--$Y$~path through $Q^\ell_1 \cup G[B_\ell] \cup b_\ell P_\ell$. 
		Since $Q'_m$ avoids $B'_\ell \supseteq B_G(B_\ell,d)$, we have $d_G(Q'_m, P'_\ell) \geq d$. Hence, $\{Q'_m\} \cup \{P'_{\ell_i} : i \leq m+1\} \cup \{P''_j : j \in [k]\setminus \{\ell_1, \dots, \ell_{m+1}\}\}$ is a collection of $k+1$ disjoint $X$--$Y$ paths pairwise at distance at least~$d$ where $P''_j$ is some $X$--$Y$ path in $P_jb'_j \cup G[B_j] \cup b_jP_j$ and $b'_j$ is the first vertex of $P_j$ that is contained in $B_j$.
		Thus, these paths are as in~\ref{itm:cor:CoarseMengerFor2:Paths}. In this case, we terminate the construction and conclude the proof. 
		
		Otherwise, let $q$ be the first vertex of $Q'_m$ that is at distance less than~$d$ from $Q^\ell_1 \cup Q^\ell_2$, and let $W'$ be a shortest $Q'_m$--$(Q^\ell_1 \cup Q^\ell_2)$ path (of length less than~$d$), and let $i \in [2]$ such that $W'$ ends in $Q^\ell_i$. If $i = 1$, then we let \defn{$P'_\ell$} be the (unique) path in $Q^\ell_1 \cup W' \cup Q'_m$ that starts in the first vertex of $Q^\ell_1$ and ends in the last vertex of~$Q'_m$. In particular, $P'_\ell$ is a path between $X$ and $Y$. Further, we let \defn{$Q'_{m+1}$} be an $X$--$Y$ path through $Q^\ell_2 \cup G[B_\ell] \cup b_\ell P_\ell$. 
		If $i=2$, then we let \defn{$Q'_{m+1}$} be the (unique) path in $Q^\ell_2 \cup W' \cup Q'_m$ that starts in the first vertex of $Q^\ell_2$ and ends in the last vertex of~$Q'_m$, and we let \defn{$P'_\ell$} be an $X$--$Y$ path through $Q^\ell_1 \cup G[B_\ell] \cup b_\ell P_\ell$. 
		
		Since $d_G(Q^\ell_1, Q^\ell_2) \geq 2d$ and $Q'_m$ avoids $B'_\ell$, we have $d_G(P'_\ell,Q'_{m+1}) \geq d$. 
		Moreover, $\{P'_{\ell_i} : i \leq m+1\} \cup \{P_j : j \in [k]\setminus \{\ell_1, \dots, \ell_{m+1}\}\}$ is still a collection of $X$--$Y$ paths that are pairwise at distance at least~$d$ and they are all at distance at least~$d$ from $Q'_{m+1}$, except for possibly $P_{\ell_{m+2}}$. 
		\medskip
		
		If the construction of the paths $P'_{\ell_m}, Q'_m$ does not terminate for some $m < n$ (in which case we already found our desired $X$--$Y$ paths), we obtain a path $Q'_n$ and paths $P'_{\ell_i}$, for $i \leq n$. Since $P_{\ell_n}$ is a leaf of~$H$ and $H$ has no parallel edges, we have $j_{\ell_n} = \ell_n$, and hence $Q'_n$ avoids all $B_G(P_j, r)\setminus A^r_j$ for $j \in [k]\setminus \{\ell_1, \dots, \ell_{n}\}$. Therefore, $Q'_n$ together with the paths $P'_{\ell_i}$, for $i \leq n$, and the paths $P''_j$, for $j \in [k]\setminus \{\ell_1, \dots, \ell_n\}$, form a collection of $k+1$ disjoint $X$--$Y$ paths that are pairwise at distance at least~$d$ where $P''_j$ is some $X$--$Y$ path in $P_jb'_j \cup G[B_j] \cup b_jP_j$ and $b'_j$ is the first vertex of $P_j$ that is contained in $B_j$.
	\end{proof}

	\section{Concluding remarks} \label{sec:ConcludingRemarks}

	The proof of \cref{main:CoarseMenger:BoundedCycleSpace} becomes significantly simpler, and the statement of \ref{itm:CoarseMengerConj:Set} can be strengthened, if we require the sets $X, Y$ to be connected.
	
	\begin{theorem} \label{main:CoarseMenger:BoundedCycleSpace:Connected}
		Set $c(\kappa, d,k) := 129\cdot\left(\left(d+\kappatwo\right)(k-2)+d+1\right)$. For every $k, d \in \N$, for every graph~$G$ whose cycle space is generated by cycles of length at most $\kappa \in \N$, and for every connected $X, Y \subseteq V(G)$, at least one of the following statements holds:
		\begin{enumerate}[label=\rm{(\roman*)}]
			\item \label{itm:CoarseMenger:CS:Connected:Paths} There are $k$ disjoint $X$--$Y$ paths in $G$ that are pairwise at distance at least $d$ from each other.
			\item \label{itm:CoarseMenger:CS:Connected:Ball} There is a vertex $z \in V(G)$ such that $B_G(z, c(\kappa,d,k))$ intersects every $X$--$Y$ path.
		\end{enumerate}
	\end{theorem}
	
	\begin{proof}
		Set $D := \left(d+\kappatwo\right)(k-2)+d+1$. By \cref{thm:CoarseMengerFor2}, there exists either a vertex $z \in V(G)$ such that $B_G(z, 129\cdot D)$ intersects every $X$--$Y$ path, or there are two $X$--$Y$ paths $P, Q$ such that $d_G(P, Q) \geq D$. Since the vertex~$z$ in the former case is as in~\ref{itm:CoarseMenger:CS:Connected:Ball}, we may assume the latter. 
		
		Let $P = p_0 \dots p_n$ and $Q = q_0 \dots q_m$, and set $P_0 := P$. Since $X$ and $Y$ are connected, there is a $p_0$--$q_0$ path $W_X$ in $G[X]$ and a $p_n$--$q_m$ path $W_Y$ in $G[Y]$. 
		For every $i \leq k-1$, set $d_i := d\cdot i+\kappatwo(i-1)$ and let $w^i_X$, $w^i_Y$ be the last vertices of $W_X$ and $W_Y$, respectively, with $d_G(w^i_A, P) = d_i+1$ for $A \in \{X,Y\}$. 
		Then $w^i_XW_Xq_0 \cup Q \cup q_mW_Yw^i_Y$ is connected and avoids $B_G(P, d_i)$. Hence, there is a component $C_i$ of $G-B_G(P, d_i)$ containing it.
		By \cref{lem:kappa/2NhoodIsConnected}, $C_i\left[\partial_G C_i, \kappatwo\right]$ is connected, and hence, since $w^i_X, w^i_Y \in \partial_G C_i$, there exists a $w^i_X$--$w^i_Y$ path $P'_i$ in $C_i\left[\partial_G C_i, \kappatwo\right]$. As $w^i_X \in X$ and $w^i_Y \in Y$, the path $P'_i$ contains an $X$--$Y$~path $P_i$. Since clearly $d_G(P_i, P_j) \geq (d_j+1) - \left(d_i+1 + \kappatwo\right) \geq d\cdot(j-i) + \kappatwo(j-i) - \kappatwo \geq d$ for $j > i$ by construction, the paths $P_0, \dots, P_{k-1}$ are as in~\ref{itm:CoarseMenger:CS:Connected:Paths}.
	\end{proof}
	
	One motivation for the coarse Menger conjecture (\cref{conj:CoarseMenger}) came from the hope that it might be useful for proving the `fat minor conjecture' of Georgakopoulos and Papasoglu~\cite{GP23}:\footnote{See e.g.\ \cite{GP23} for definitions.}
	
	\begin{conjecture}[Fat minor conjecture] \label{conj:FatMinor}
		For every graph $X$ there exists a function $f: \N \to \N^2$ such that the following holds for every graph $G$ and $K \in \N$: If $G$ does not contain $X$ as a $K$-fat minor, then $G$ is $f(K)$-quasi-isometric to a graph with no $X$ minor.
	\end{conjecture}
	
	\cref{conj:FatMinor} has been disproved by Davies, Hickingbotham, Illingworth, and McCarty~\cite{DHIM}. Since then, further counterexamples have been obtained by Albrechtsen, Distel, and Georgakopoulos~\cite{ADGSmallCounterexFatMinorConj} and by Albrechtsen and Davies~\cite{ADWeakCounterex}, showing that \cref{conj:FatMinor} fails even for certain small graphs $X$ such as $K_{2,2,2}$~\cite{ADGSmallCounterexFatMinorConj}, and that it fails even if we weaken its conclusion and only ask for a quasi-isometry to some graph excluding some (possibly bigger) graph~$X'$ (depending only on~$X$) as a minor~\cite{ADWeakCounterex}. 
	
	All these counterexamples build on a construction of Nguyen, Scott, and Seymour~\cite{CoarseMengerCounterex,NewCounterexToCoarseMenger} which they used to disprove the coarse Menger conjecture (\cref{conj:CoarseMenger}). Consequently, the cycle spaces of 
	all graphs $G$ that were used to disprove \cref{conj:FatMinor} are not generated by cycles of bounded length. Hence, the following special case of \cref{conj:FatMinor} is still open:
	
	\begin{conjecture} \label{conj:FatMinors:CS}
		For every graph $X$ there exists a function $f: \N^2 \to \N^2$ such that the following holds for every graph $G$ whose cycle space is generated by cycles of length at most $\kappa$: If $G$ does not contain $X$ as a $K$-fat minor, for some $K \in \N$,  then $G$ is $f(\kappa, K)$-quasi-isometric to a graph with no $X$ minor.
	\end{conjecture}
	
	A particularly interesting special case of \cref{conj:FatMinors:CS} would be the coarse Kuratowski conjecture: Is it true that if a graph $G$ does not contain $K_5$ and $K_{3,3}$ as a $K$-fat minor and its cycle space is generated by cycles of length at most~$\kappa$, then $G$ is $f(\kappa, K)$-quasi-isometric to a planar graph?
	
	It is know that \cref{conj:FatMinors:CS} is true in the special case of locally finite Cayley graphs of finitely presented groups: Albrechtsen and Hamann~\cite{AHAsymptoticGrids} proved \cref{conj:FatMinors:CS} for all planar graphs~$X$ if $G$ is locally finite and quasi-transitive, and MacManus~\cite{M24+} proved \cref{conj:FatMinors:CS}
	for all non-planar graphs~$X$ if $G$ is a locally finite Cayley graph of a finitely presented group.

	\section*{Acknowledgements}
	
	We thank Agelos Georgakopoulos for drawing our attention to the problem of coarse Menger for graphs whose cycle space is generated by cycles of bounded length.
	We thank James Davies for his thorough input on the introduction.

	\printbibliography
	
\end{document}